\documentclass[11pt]{article}
\usepackage{amssymb,amsmath,latexsym}
\usepackage{amsthm}
\usepackage[shortlabels]{enumitem}
\usepackage{caption}
\usepackage{xcolor}
\usepackage{float}
\usepackage{graphicx}
\usepackage{arydshln}
\usepackage[utf8]{inputenc}
\usepackage[english]{babel}
\usepackage[T2A]{fontenc}
\usepackage{tikz}

\usepackage{tabularx} 
\usepackage{ytableau}
\usepackage{colortbl} 
\usepackage{url}
\usepackage{breakurl}
\usepackage[colorlinks=true,citecolor=blue,linkcolor=blue,urlcolor=blue,bookmarks,bookmarksopen,bookmarksdepth=2,backref=page,breaklinks]{hyperref}

\usepackage{bookmark}
\bookmarksetup{
	numbered, 
	open,
}

\renewcommand*{\backref}[1]{}
\renewcommand*{\backrefalt}[4]{%
	\ifcase #1 (Not cited.)%
	\or        (Cited on page~#2.)%
	\else      (Cited on pages~#2.)%
	\fi}

\theoremstyle{plain}
\newtheorem{theorem}{Theorem}[section]
\newtheorem{lemma}[theorem]{Lemma}
\newtheorem{proposition}[theorem]{Proposition}

\newtheorem{corollary}[theorem]{Corollary}

\theoremstyle{definition}

\newtheorem{example}[theorem]{Example}

\newtheorem{definition}[theorem]{Definition}
\newtheorem{remark}[theorem]{Remark}

\newtheorem{notation}[theorem]{Notation}

\numberwithin{theorem}{subsection}
\numberwithin{equation}{section}
\numberwithin{table}{section}
\numberwithin{figure}{section}
\numberwithin{openproblem}{section}

\begin{document}
\title{A bijection between peakless Motzkin paths and LR tableaux}

\author{Milan Tenn\vspace{0.4cm} \ \\
Swarthmore College \\\tt mtenn1@swarthmore.edu\vspace{0.4cm}
}

\date{\today}
\maketitle
\abstract{We prove the conjecture of Donnelly et al. that 
a certain class of Littlewood-Richardson tableaux are equinumerous with peakless Motzkin paths of length $n$. Furthermore, we construct an explicit bijection between this class of tableaux and peakless Motzkin paths of length $n$ for all $n\ge 1$.}\ \\[2mm]


\thispagestyle{empty}
\section{Introduction}
Young tableaux are a well-studied class of combinatorial objects. One particular type of tableau which is of interest is Littlewood-Richardson tableaux, or LR tableaux. Both Young tableaux in general and LR tableaux in particular are highly related to representation theory. For instance, LR tableaux correspond to the Littlewood-Richardson rule. This rule determines the multiplication of Schur polynomials, which form a basis of the symmetric polynomials and which frequently appear in problems relating to representation theory. These classes of tableaux have been investigated in many papers, including~\cite{youngtableauexpository,leeuwen2001,pak2005}.

This paper will investigate a class of tableaux called Fibonacci ribbons. Fibonacci ribbons are a class of skew-shaped semistandard tableaux which were introduced in~\cite{tableauconjecturepaper} in relation to representations of special linear Lie algebras. Fibonacci ribbons which are also LR tableaux are called ballot-admissible. Based on some empirical numerical evidence, it was conjectured in~\cite{tableauconjecturepaper} that a certain class of ballot-admissible Fibonacci ribbons are in bijection with a class of lattice walks called peakless Motzkin paths. This claim was also investigated with additional numerical evidence in~\cite{malonepaper}.  It is known that certain other classes of tableaux are related to classes of lattice walks, for instance as shown in~\cite{drube2017,watanabe1988}.

Peakless Motzkin paths are a themselves a well-studied class of lattice walks. These paths and some variants have been investigated significantly in papers including~\cite{evanspaper,grandgraycodes,markedlevelsteps,asamoahrnabijectionorigin,grandtrees}. This investigation has demonstrated a number of interesting applications for peakless Motzkin paths, including to the enumeration of RNA secondary structures~\cite{asamoahearlypaper}. As a result, the conjecture that peakless Motzkin paths are equinumerous with ballot-admissible Fibonacci ribbons creates a new connection between the mathematics of RNA secondary structure prediction and combinatorial representation theory.

In this paper, we confirm that the conjecture posed in~\cite{tableauconjecturepaper} is correct. We first prove that the class of ballot-admissible Fibonacci ribbons described in~\cite{tableauconjecturepaper} is in fact equinumerous with peakless Motzkin paths. In addition, we construct an explicit bijection between these ballot-admissible Fibonacci ribbons and corresponding peakless Motzkin paths.

This paper is organized in the following way. In Section~\ref{sec:background}, we introduce some background and preliminary results about both peakless Motzkin paths and Fibonacci ribbons. In Section~\ref{sec:bijection}, we construct an explicit bijection between ballot-admissible ribbons and peakless Motzkin paths. Finally, in Section~\ref{sec:future}, we discuss conclusions and ideas regarding future work.

\section{Background and Preliminaries}\label{sec:background}
\subsection{Skew-Shaped Tableaux}
In this subsection, we introduce definitions for skew-shaped semistandard tableaux and Littlewood-Richardson tableaux. We also define a matching between entries of Littlewood-Richardson tableaux.

\begin{definition}
    Let $[n]$ be the set $\{1,2,\ldots,n\}$. 
\end{definition}

\begin{definition}
    For $n_1 < n_2$, let $[n_1,n_2]$ be the set $\{n_1,n_1+1,\ldots,n_2\}$. 
\end{definition}

\begin{definition}
    Let $P\in \mathbb{Z}^n$, where $P=(p_1,\ldots,p_n)$ for $p_1 \ge p_2\ge \cdots \ge p_n\ge 0$. We define the Young diagram for $P$ as having $n$ rows, where the $i$th row has $p_i$ left-justified boxes. 
\end{definition}

\begin{definition}
    Consider some $P,Q\in\mathbb{Z}^n$ where both $P$ and $Q$ correspond to Young diagrams and where for all $i\in [n]$, $p_i > q_i$. We define the shape $P/Q$ by taking the Young diagram for $P$ and removing the boxes corresponding to the Young diagram for $Q$. We do so in each row $i$ by removing the first $q_i$ boxes from the left.
\end{definition}

\begin{definition}
    Given some $P/Q$, we define a skew-shaped tableau with shape $P/Q$ by assigning some positive integer value to every box in $P/Q$.
\end{definition}

\begin{definition}
    We call a tableau semistandard if its rows are weakly increasing from left to right, and its columns are strictly increasing from top to bottom.
\end{definition}

\begin{definition}
    The linear form of a tableau $t$ is the word formed from the entries of $t$ read from top to bottom and right to left.
\end{definition}

\begin{example}
    The linear form of the semistandard tableau \begin{ytableau}
            \none & \none & 1\\
            1 & 1 & 2\\
            3
    \end{ytableau} is $(12)(1)(13)=12113$.
\end{example}

\begin{definition}
    Let $w$ be some word of length $k$ formed from letters in $[n]$. For $1\le i\le k$, let $w(i)$ be the $i$-th letter of $w$ read from left to right.
\end{definition}

\begin{definition}
    Let $w$ be some word of length $k$ formed from letters in $[n]$. Let $S\subseteq [k]$ where $S=\{s_1,\ldots,s_j\}$ and $s_1 < \cdots < s_j$. Then, $w(S)=w(s_1)\cdots w(s_j)$.
\end{definition}

\begin{definition}
    Given some $i\in\mathbb{Z}_+$ and some word $w$, let $\mu_i(w)$ be the number of occurences of $i$ in $w$. For some tableau $t$, let $\mu_i(t)=\mu_i(w)$ where $w$ is the linear form of $t$.
\end{definition}

\begin{definition}
    Let $t$ be some skew-shaped semistandard tableau with entries in $[n]$. Let $w$ be the linear form of $t$. We say that $t$ is a Littlewood-Richardson tableau if for all prefixes $w'$ of $w$ and all $i\in [n-1]$, $\mu_i(w') \ge \mu_{i+1}(w')$.
\end{definition}

\begin{definition}\label{def:matching}
    Let $w$ be some word with letters in $[n]$. Let $w(x)=i+1$ for some $i\in [n-1]$. Let $y=b_i(w,x)$ be defined as the maximal value such that $y \le x$ and for $w'=w([y,x])$, $\mu_i(w')=\mu_{i+1}(w')$.
\end{definition}

\begin{remark}
    It is easy to see that for any $y=b_i(w,x)$, $w(y)=i$. 
    Since it is possible that no such $w'$ exists, $b_i$ induces a partial function from $i+1$ letters to $i$ letters in any word $w$. 
\end{remark}

\begin{proposition}\label{prop:binjective}
    Let $w$ be some word with letters in $[n]$. For all $i\in [n-1]$, $x\mapsto b_i(w,x)$ is injective.
\end{proposition}

\begin{proof}
    Suppose for the sake of contradiction that for $x_1 < x_2$, $b=b_i(w,x_1)=b_i(w,x_2)$. Then, by definition, for $w_1=w([b,x_1])$ and $w_2=w([b,x_2])$, $\mu_i(w_1)=\mu_{i+1}(w_1)$ and $\mu_i(w_2)=\mu_{i+1}(w_2)$. Then, since $x_1 < x_2$, we find that $[b,x_2]-[b,x_1]=[x_1+1,x_2]$ for $x_1 + 1 \le x_2$. 
    
    Since $\mu_i(w_1)=\mu_{i+1}(w_1)$ and $\mu_i(w_2)=\mu_{i+1}(w_2)$, $\mu_i(w([x_1+1,x_2]))=\mu_{i+1}(w([x_1+1,x_2]))$. However, $b \le x_1 < x_1+1$, and $b$ is the maximal value where $w([b,x_2])$ has an equal number of $i$ and $i+1$ letters. As such, this is impossible, and we conclude that for $x_1 < x_2$, $b_i(w,x_1)\ne b_i(w,x_2)$. This is sufficient to show that $x\mapsto b_i(w,x)$ is injective as desired.
\end{proof}

\begin{proposition}\label{prop:bsufficient}
    Let $w$ be some word with letters in $[n]$. Let $w(x)=i+1$ for some $i\in [n-1]$. Suppose that for $w'=w([y,x])$, $\mu_i(w')\ge \mu_{i+1}(w')$. Then, $b_i(w,x)$ is well-defined, and $y\le b_i(w,x)$.
\end{proposition}

\begin{proof}
    For $z\le x$, let $w_z = w([z,x])$. Let $f(z)=\mu_i(w_z)-\mu_{i+1}(w_z)$. Then, since $w(z-1)$ gives us at most one occurence of either $i$ or $i+1$, $|f(z-1)-f(z)|\le 1$. This means that given $z_1 < z_2$, $f$ must achieve all values between $f(z_1)$ and $f(z_2)$ in $[z_1,z_2]$. 
    
    Now, we note that $f(x)=-1$ since $w(x)=i+1$ by definition. Finally, we know that $w_y=w'$ has $f(y)\ge 0$. From this, we conclude that there exists some $z\in [y,x-1]$ such that $f(z)=0$. This is sufficient to conclude that $b_i(w,x)$ is well-defined. By maximality of $b_i(w,x)$, we know that $b_i(w,x)\ge z\ge y$.
\end{proof}

\begin{theorem}\label{thm:lrcondition}
    Let $t$ be some skew-shaped semistandard tableau with entries in $[n]$. Let $w$ be the linear form of $t$. Then, $t$ is a Littlewood-Richardson tableau iff $\forall i\in [n-1]$, $x\mapsto b_i(w,x)$ is a function. 
\end{theorem}

\begin{proof}
    First, suppose that for all $i\in [n-1]$, $x\mapsto b_i(w,x)$ is a function. Then, let $w'$ be some prefix of $w$. Since $b_i(w,x) < x$, we see that each $b_i$ induces a function from $i+1$ letters to preceding $i$ letters. As such, each $b_i$ induces a function from $i+1$ letters in $w'$ to $i$ letters in $w'$. Since each $b_i$ induces an injective function by Proposition~\ref{prop:binjective}, we conclude that $\mu_{i+1}(w')\le \mu_i(w')$ for all $i\in [n-1]$ as desired. As such, $t$ is a Littlewood-Richardson tableau.

    Next, suppose that $t$ is a Littlewood-Richardson tableau. Then, let $w(x)=i$. Let $w'=w([1,x])$ be the prefix of $w$ which ends on $w(x)$. By definition of a Littlewood-Richardson tableau, $\mu_i(w')\ge \mu_{i+1}(w')$. By Proposition~\ref{prop:bsufficient}, it follows that $b_i(w,x)$ is well-defined as desired. Thus, we see that $x\mapsto b_i(w,x)$ is a function for all $i\in [n-1]$ iff $t$ is a Littlewood-Richardson tableau.
\end{proof}

\begin{remark}
    This is a generalization of the matching between $i$ and $i+1$ entries given in~\cite[Section 3.1]{leeuwen2001}. Instead of a matching induced by Dyck words of $i$ and $i+1$ letters, we have Dyck prefixes.
\end{remark}

\subsection{Fibonacci Ribbons}
In this section, we define Fibonacci ribbons to be $\text{Fib}(3,k)$ ribbons as given in~\cite{tableauconjecturepaper}.

\begin{definition}
    Let $n\ge 1$. For $P=(n,n-1,n-3,\ldots)$ and $Q=(n-2,n-4,\ldots)$ where we let $P$ and $Q$ decrease to zero rather than becoming negative, we call the shape $P/Q$ the Fibonacci ribbon shape of length $n$.
\end{definition}

\begin{example}
    The Fibonacci ribbon shapes of length 1, 2, 3 and 4 are as shown below.
    \begin{center}
        \begin{ytableau}
            ~
        \end{ytableau} \begin{ytableau}
            ~ & ~\\
            ~
        \end{ytableau} \begin{ytableau}
            \none & ~ & ~\\
            ~ & ~\\
        \end{ytableau} \begin{ytableau}
            \none & \none & ~ & ~\\
            ~ & ~ & ~\\
            ~
        \end{ytableau}
    \end{center}
\end{example}

\begin{definition}
    Let $n\ge 1$. For $P=(n,n,n-2,n-4,\ldots)$ and $Q=(n-1,n-3,\ldots)$ where we let $P$ and $Q$ decrease to zero rather than becoming negative, we call the shape $P/Q$ the Fibonacci dual ribbon shape of length $n$.
\end{definition}

\begin{example}
    The Fibonacci dual ribbon shapes of length 1, 2, 3 and 4 are as shown below.
    \begin{center}
        \begin{ytableau}
            ~\\
            ~
        \end{ytableau} \begin{ytableau}
            \none & ~\\
            ~ & ~
        \end{ytableau} \begin{ytableau}
            \none & \none & ~\\
            ~ & ~ & ~\\
            ~
        \end{ytableau} \begin{ytableau}
            \none & \none & \none & ~\\
            \none & ~ & ~ & ~\\
            ~ & ~
        \end{ytableau}
    \end{center}
\end{example}

\begin{proposition}
    The Fibonacci ribbon shape of length $n$ has exactly $n$ columns. If $i\ge 1$ is odd, then the $i$-th column from right to left has 1 box. If $i$ is even, then the $i$-th column has 2 boxes.
\end{proposition}

\begin{proof}
    This is given in~\cite{tableauconjecturepaper}.
\end{proof}

\begin{proposition}
    The Fibonacci dual ribbon shape of length $n$ has exactly $n$ columns. If $i\ge 1$ is odd, then the $i$-th column from right to left has 2 boxes. If $i$ is even, then the $i$-th column has 1 box.
\end{proposition}

\begin{proof}
    This can be easily verified in the same way as for the Fibonacci ribbon shape.
\end{proof}

\begin{remark}
    Removing the leftmost column from a ribbon or dual ribbon shape of length $n$ yields the ribbon or dual ribbon shape of length $n-1$ respectively. Removing the rightmost column from the ribbon shape of length $n$ yields the dual ribbon of length $n-1$ and vice versa.
\end{remark}

\begin{definition}
    A Fibonacci ribbon or dual ribbon of length $n$ is a skew-shaped semistandard tableau of the Fibonacci ribbon or dual ribbon shape of length $n$ with entries in $\{1,2,3\}$. 
\end{definition}

\begin{definition}
    Let $r$ be some ribbon or dual ribbon of length $n$. For $1\le i\le n$, we define $r(i)$ to be the $i$-th column from right to left of $r$.
\end{definition}

\begin{definition}
    Let $r$ be some ribbon or dual ribbon of length $n$. For $1\le x\le y\le n$, let $r([x,y])$ be the ribbon formed from the columns $x\le i\le y$ of $r$.
\end{definition}

\begin{example}
    If $t$ is the tableau \begin{ytableau}
            \none & \none & 1\\
            1 & 1 & 2\\
            3
    \end{ytableau} then $t(1)=(12)$, $t(2)=(1)$ and $t(3)=(13)$. Furthermore, we observe that $t([1,2])=(12)(1)$.
\end{example}

\begin{proposition}\label{prop:ribboncondition}
    Let $r$ be a tableau of the Fibonacci ribbon or dual ribbon shape of length $n$. Let $w$ be the linear form of $r$. Then, $r$ is a Fibonacci ribbon or dual ribbon iff the following holds:
    \begin{enumerate}
        \item[(i)] For all $1\le i\le n$, either $r(i)=(j)$ or $r(i)$ has entries which are exactly $\{1,2,3\}-\{j\}$ in ascending order for some $j\in [3]$.
        \item[(ii)] The pattern $123$ does not occur in $w$.  
    \end{enumerate}
\end{proposition}

\begin{proof}
    Since $r$ is a tableau of the Fibonacci ribbon or dual ribbon shape, $r$ is a Fibonacci ribbon or dual ribbon iff $r$ is semistandard. We will show that (i) and (ii) hold if $r$ is semistandard. In addition, we will show that if $r$ is not semistandard and (i) holds, then (ii) does not hold.
    
    Suppose that $r$ is semistandard. Then, the entries of $r$ must be strictly increasing from top to bottom in each column. Since the allowed entries are $\{1,2,3\}$ and there are either 1 or 2 boxes in each column, (i) evidently holds. In addition, say that the pattern 123 occurs in $w$. Then, observe that this must correspond to successive columns of the form $(12)(3)$ or $(1)(23)$. In either case, there is a row of $r$ that decreases from left to right, which is impossible, as $r$ is semistandard. Thus, 123 cannot occur in $w$.

    Suppose that $r$ is not semistandard and (i) holds. Then, since (i) holds, columns must be strictly increasing from top to bottom. Since $r$ is not semistandard, there is some row where the entries decrease from left to right. This corresponds to columns of $r$ which appear in one of the two shapes below:
    \begin{center}
        \begin{ytableau}
            a & b\\
            c
        \end{ytableau} \begin{ytableau}
            \none & c\\
            a & b
        \end{ytableau}
    \end{center}

    In either case, $a > b$, so $a\ne 1$. If we consider the diagram on the left, it must then be the case by (i) that $(ac)=(23)$. In addition, $b < a$ so $b=1$, and we have $(1)(23)$ in successive columns, so (ii) does not hold. In the diagram on the right, $b\ne 3$ since $b < a$, so $(cb)=(12)$ and $a>b$ so $a=3$. We again find a $(12)(3)$ pattern as desired.
\end{proof}

\begin{definition}
    Let $r$ be a Fibonacci ribbon or dual ribbon. We define $r^*$, the dual of $r$, as follows:
    \begin{enumerate}
        \item[(i)] If $r(i)=(j)$ then $r^*(i)$ has entries $\{1,2,3\}-\{4-j\}$ increasing from top to bottom.
        \item[(ii)] If $r(i)$ has entries $\{1,2,3\}-\{j\}$ then $r^*(i)=(4-j)$. 
    \end{enumerate}
\end{definition}

\begin{example}
    The dual of  \begin{ytableau}
            \none & \none & 1\\
            1 & 1 & 2\\
            3
    \end{ytableau} is \begin{ytableau}
            \none & 1 & 1\\
            2 & 2\\
        \end{ytableau}.
\end{example}

\begin{proposition}
    The map $r\mapsto r^*$ induces a bijection between Fibonacci ribbons and dual ribbons of length $n$ for all $n\ge 1$.
\end{proposition}

\begin{proof}
    By definition, it is evident that $(r^*)^*=r$. As such, it is sufficient to show that $r\mapsto r^*$ maps ribbons to dual ribbons and vice versa.

    Let $r$ be a Fibonacci ribbon of length $n$. By definition, it is clear that $r^*$ has $n$ columns, with 1 entry when $r$ has 2 entries, and vice versa. As such, $r^*$ is a tableau of the Fibonacci dual ribbon shape of length $n$. Furthermore, we note that by definition, $r\mapsto r^*$ maps $(1)(23)$ to $(12)(3)$ and vice versa. As such, since the linear form of $r$ has no 123 patterns, neither does the linear form of $r^*$. Finally, we note that by definition, the columns of $r^*$ have increasing entries from top to bottom. Then, by Proposition~\ref{prop:ribboncondition}, we conclude that $r^*$ is a Fibonacci dual ribbon of length $n$ as desired.

    An analogous argument holds for Fibonacci dual ribbons of length $n$.
\end{proof}

\subsection{Ballot-Admissibility and Bonds}
\begin{definition}
    A Fibonacci ribbon or dual ribbon which is a Littlewood-Richardson tableau is called a ballot-admissible ribbon or dual ribbon.  
\end{definition}

\begin{proposition}
    Let $r$ be a ribbon or dual ribbon. Then, $b_i$ from Definition~\ref{def:matching} induces an injective partial function from columns of $r$ with an $i+1$ entry to columns of $r$ with an $i$ entry.
\end{proposition}

\begin{proof}
    By definition, $b_i$ is a partial function from $i+1$ letters to $i$ letters in the linear form of $r$. Since by Proposition~\ref{prop:ribboncondition}, there may exist at most one $i+1$ entry in each column of $r$, $b_i$ induces a partial function from columns of $r$ with an $i+1$ entry to columns with an $i$ entry. Furthermore, as $b_i$ is injective, and by Proposition~\ref{prop:ribboncondition}, any given column has at most one $i$ entry, the partial function induced by $b_i$ remains injective.
\end{proof}

\begin{definition}
    Given some ribbon or dual ribbon $r$, we denote the partial function induced by each $b_i$ as $c\mapsto b_i(r,c)$. Note that by Definition~\ref{def:matching} for $b_i$, $b_i(r,c) \le c$ is the maximal column in $r$ such that $\mu_i(r([b_i(c), c]))=\mu_{i+1}(r([b_i(c), c]))$.
\end{definition}

\begin{notation}
    Given a fixed ribbon $r$, we may refer to this partial function as $c\mapsto b_i(c)$.
\end{notation}

\begin{proposition}\label{prop:ballotcondition}
    Let $r$ be a ribbon or dual ribbon. Then, $r$ is ballot-admissible iff each partial function $c\mapsto b_i(c)$ is a function.
\end{proposition}

\begin{proof}
    Since each partial function $c\mapsto b_i(c)$ is induced by $b_i$ as given in Definition~\ref{def:matching}, this follows immediately from Theorem~\ref{thm:lrcondition}.
\end{proof}

\begin{proposition}\label{prop:bondentryamounts}
    Let $r$ be a ballot-admissible Fibonacci ribbon or dual ribbon. Then, $c$ is the minimal column $ \ge d$ such that $\mu_i(r([d,c]))=\mu_{i+1}(r([d,c]))$ iff $b_i(c)=d$.
\end{proposition}

\begin{proof}
    Let $b_i(c)=d$. Then, by definition, $d$ is the maximal value such that $\mu_i(r([d,c]))=\mu_{i+1}(r([d,c]))$. Next, suppose for the sake of contradiction that for $e < c$, $\mu_i(r([d,e]))=\mu_{i+1}(r([d,e]))$. We note that $[d,c]-[d,e]=[e+1,c]$. Since $\mu_i(r([d,c]))=\mu_{i+1}(r([d,c]))$ and $\mu_i(r([d,e]))=\mu_{i+1}(r([d,e]))$, it then follows that $\mu_i(r([e+1,c]))=\mu_{i+1}(r([e+1,c]))$, which is impossible by maximality of $d$, as $e+1 > d$. Thus, we have a contradiction, and by contradiction, $c$ is the minimal column where $\mu_i(r([d,c]))=\mu_{i+1}(r([d,c]))$ as desired.

    Next, suppose that $c$ is the minimal column $\ge d$ such that $\mu_i(r([d,c]))=\mu_{i+1}(r([d,c]))$. Then, $b_i(c)$ is defined as the maximal $e\le c$ such that $\mu_i(r([e,c]))=\mu_{i+1}(r([e,c]))$. Suppose for the sake of contradiction that $e > d$. Then, since we have $[d,c]-[e,c]=[d,e-1]$, $\mu_i(r([d,c]))=\mu_{i+1}(r([d,c]))$, and $\mu_i(r([e,c]))=\mu_{i+1}(r([e,c]))$, it follows that $\mu_i(r([d,e-1]))=\mu_{i+1}(r([d,e-1]))$, but since $e-1 < c$, this is impossible by minimality of $c$. As such, we have a contradiction. By contradiction, $e\le d$, which is sufficient by maximality of $e$ to conclude that $d=e=b_i(c)$.
\end{proof}

\begin{corollary}\label{cor:bondentryamounts}
    Let $r$ be a ballot-admissible ribbon. Let $r(c)$ have an $i$ entry. Let $c \le d$. If $\mu_i(r([c,d]))\le \mu_{i+1}(r([c,d]))$, then there exists some $e\le d$ such that $c=b_i(e)$.
\end{corollary}

\begin{proof}
    For $z\ge c$, let $r_z=r([c,z])$ and let $f(z)=\mu_i(r_z)-\mu_{i+1}(r_z)$. Since $r(c)$ has an $i$ entry, it follows that $f(c)\ge 0$. In addition, by our assumption, $f(d)\le 0$. Finally, we note that $|f(z+1)-f(z)|\le 1$ since any given column may have at most one $i$ and one $i+1$ entry. As such, $f$ must achieve intermediate values, and there exists some $c\le z\le d$ such that $f(z)=0$.

    Now, let $e\le z$ be the minimal column $\ge c$ such that $f(e)=0$. By Proposition~\ref{prop:bondentryamounts}, $b_i(e)=c$. 
\end{proof}

\begin{proposition}\label{prop:naturalmatch}
    Let $r$ be a ballot-admissible Fibonacci ribbon or dual ribbon of length $n$. Then, each map $c\mapsto b_i(c)$ obeys the following properties:
    \begin{enumerate}
        \item[(i)] Suppose that $b_i(c_0)=d_0$ and $b_i(c_1)=d_1$. If $c_0 < c_1$, either $d_1 < d_0$ or $c_0 < d_1$.
        \item[(ii)] Suppose that $b_i(c_0)=d_0$ and that for $d_0\le d_1 \le c_0$, $r(d_1)$ has an $i$ entry. Then, there exists some $c_1 \le c_0$ such that $b_i(c_1)=d_1$. 
    \end{enumerate}
\end{proposition}

\begin{proof}
    We first show (i). For $c_0 < c_1$, let $b_i(c_0)=d_0$ and $b_i(c_1)=d_1$. Suppose for the sake of contradiction that $d_0 \le d_1\le c_0$. By injectivity of $b_i$, $d_0 \ne d_1$. In addition, if $c_0$ has an $i$ entry, then $\mu_i(r(c_0))=\mu_{i+1}(r(c_0))$ and so $b_i(c_0)=c_0$. By injectivity of $b_i$ again, $d_1\ne c_0$. As such, $d_0 < d_1 < c_0$. Now, by Corollary~\ref{cor:bondentryamounts}, since $d_1=b_i(c_1)$ for $c_1 > c_0$, $\mu_i(r([d_1,c_0])) > \mu_{i+1}(r([d_1,c_0]))$. In addition, since $d_0=b_i(c_0)$, $\mu_i(r([d_0,c_0]))=\mu_{i+1}(r([d_0,c_0]))$. Since $d_0 < d_1$, $\mu_i(r([d_0,d_1-1]))\le \mu_{i+1}(r([d_0,d_1-1]))$. By Corollary~\ref{cor:bondentryamounts}, it follows that $d_0=b_i(e)$ for some $e\le d_1-1 < c_0$, which is impossible by injectivity of $b_i$. As such, we have a contradiction, and by contradiction, either $d_1 < d_0$ or $c_0 < d_1$.

    We next show (ii). Suppose that $b_i(c_0)=d_0$ and $r(d_1)$ has an $i$ entry for $d_0\le d_1\le c_0$. Then, since $b_i(c_0)=d_0$, $\mu_i(r([d_0,c_0]))=\mu_{i+1}(r([d_0,c_0]))$. Now, if $d_1=d_0$ then $d_1=b_i(c_1)$ and we are done. If $d_0 < d_1 \le c_0$, then since $d_0=b_i(c_0)$ for $c_0 > d_1-1$, $\mu_i(r([d_0,d_1-1])) > \mu_{i+1}(r([d_0,d_1-1]))$ by Corollary~\ref{cor:bondentryamounts}. In addition, as $d_0=b_i(c_0)$, $\mu_i(r([d_0,c_0]))=\mu_{i+1}(r([d_0,c_0]))$. 
    
    As $[d_0,c_0]-[d_0,d_1-1]=[d_1,c_0]$, we conclude that $\mu_i(r([d_1,c_0])) < \mu_{i+1}(r([d_1,c_0]))$. Then, by Corollary~\ref{cor:bondentryamounts}, $d_1=b_i(c_1)$ for some $c_1\le c_0$ as desired.
\end{proof}

\begin{definition}
    Let $r$ be some ribbon or dual ribbon. If $c < d$ and $c=b_i(d)$ for some $i$, then we say that $c$ is bonded in $r$, or that $c$ and $d$ are bonded in $r$.
\end{definition}

\begin{remark}
    Proposition~\ref{prop:naturalmatch} allows us to clearly interpret each $b_i$ as a noncrossing matching. This is reminiscent of the correspondence between RNA secondary structure bonds and peakless Motzkin path steps, which is why we refer to this correspondence as a bond.
\end{remark}

\begin{proposition}
    Let $r$ be some ribbon or dual ribbon. Let $c$ be some column of $r$. Then,
    \begin{enumerate}
        \item[(i)] There exists at most one column $d < c$ where $d$ and $c$ are bonded.
        \item[(ii)] There exists at most one column $d > c$ where $d$ and $c$ are bonded.
    \end{enumerate}
\end{proposition}

\begin{proof}
    For any column $c$, by Proposition~\ref{prop:ribboncondition}, there exists at most one entry $i+1$ where no $i$ entry exists in $r(c)$. For any other value of $i$, $r(c)$ has both an $i$ and $i+1$ entry, so $b_i(c)=c$. As such, if $d,e < c$ and $d=b_{i_1}(c)$ and $e=b_{i_2}(c)$, it must be the case that $i_1=i_2$. So, $d=e$ by injectivity of each $b_i$. An analogous argument holds for the case where $d > c$.
\end{proof}

\begin{definition}
    Let $r$ be some ribbon or dual ribbon. If $r(c)$ has an $i$ entry and no $i+1$ entry, but $c\notin \text{Im}(b_i(r,\_))$, then we say that $c$ is unbonded, or $i$-unbonded, in $r$.
\end{definition}

\begin{proposition}\label{prop:bondchangecrit}
    Let $r_0$ and $r_1$ be ballot-admissible ribbons of the same length with identical entries except in some column $c'$. Then, the following holds:
    \begin{enumerate}
        \item[(i)] For $c < d < c'$, if $c$ and $d$ are bonded in $r_0$, then $c$ and $d$ are bonded in $r_1$.
        \item[(ii)] For $c' < c < d$, if $c$ and $d$ are bonded in $r_0$, then $c$ and $d$ are bonded in $r_1$. 
    \end{enumerate}
\end{proposition}

\begin{proof}
    We first show (i). Suppose that $c < d < c'$ where $c$ and $d$ are bonded in $r_0$. Then, for some $i$, $d$ is the minimal column such that $\mu_i(r_0([c,d]))=\mu_{i+1}(r_0([c,d]))$. Since $c,d < c'$, we know that $\mu_i(r_1([c,d]))=\mu_{i+1}(r_1([c,d]))$ as well. In addition, for $c\le e < d$, if $\mu_i(r_1([c,e]))=\mu_{i+1}(r_1([c,e]))$ then since $e < c'$ as well, $\mu_i(r_0([c,e]))=\mu_{i+1}(r_0([c,e]))$ as well. This is impossible by minimality of $d$ in $r_0$, so $d$ must also be minimal in $r_1$, and by Proposition~\ref{prop:bondentryamounts}, $c$ and $d$ remain bonded in $r_1$.

    We next show (ii). Since $c' < c < d$, for all $e\ge c$, $r_0([c,e])=r_1([c,e])$ and so by an argument identical to the argument above, minimality of $d$ is preserved.
\end{proof}

\begin{lemma}\label{lemma:dualcounts}
    Let $r$ be a Fibonacci ribbon or dual ribbon. Then, $\mu_i(r)-\mu_{i+1}(r)=\mu_{3-i}(r^*)-\mu_{4-i}(r^*)$.
\end{lemma}

\begin{proof}
    Consider each column $c$ of $r$ and $r^*$. If $r(c)$ has both $i$ and $i+1$ then $r^*(c)=(j)$ where $4-j\ne i,i+1$, so $j\ne 3-i, 4-i$. If $r(c)$ has neither $i$ nor $i+1$ then $r(c)=(j)$ for $j\ne i,i+1$ and so $r^*(c)$ has entries of both value $4-i$ and $4-(i+1)=3-i$. In all of these cases,
    \[\mu_i(r(c))-\mu_{i+1}(r(c)) = \mu_{3-i}(r^*(c))-\mu_{4-i}(r^*(c))=0.\]

    If $r(c)=(i)$ then $r^*(c)$ has all values except for $4-i$. As such, $r^*(c)$ has a $3-i$ entry and no $4-i$ entry. Then, 
    \[\mu_i(r(c))-\mu_{i+1}(r(c)) = \mu_{3-i}(r^*(c))-\mu_{4-i}(r^*(c))=1.\]

    If $r(c)$ has entries of all values except for $i$, then $r^*(c)=(4-i)$, and we find that
    \[\mu_i(r(c))-\mu_{i+1}(r(c)) = \mu_{3-i}(r^*(c))-\mu_{4-i}(r^*(c))=-1\]

    in both cases. If $r(c)=(i+1)$ or $r(c)$ has all values except for $i+1$, analogous statements hold.
    In this way, for all columns $c$, 
    \[\mu_i(r(c))-\mu_{i+1}(r(c)) = \mu_{3-i}(r^*(c))-\mu_{4-i}(r^*(c))\]

    and so $\mu_i(r)-\mu_{i+1}(r)=\mu_{3-i}(r^*)-\mu_{4-i}(r^*)$ as desired.
\end{proof}

\begin{theorem}
    The map $r\mapsto r^*$ preserves bonds between columns and ballot-admissibility.
\end{theorem}

\begin{proof}
    Let $r$ be a Fibonacci ribbon or dual ribbon. 
    Suppose that $c < d$ and $c,d$ are bonded in $r$. We will show that $c,d$ are bonded in $r^*$. Since $c < d$, we know that for some $i$, $r(c)$ has an $i$ entry and no $i+1$ entry. Furthermore, $d$ is the minimal column such that $\mu_i(r([c,d]))-\mu_{i+1}(r([c,d]))=0$ by Proposition~\ref{prop:bondentryamounts}. 
    
    Now, since $r(c)$ has an $i$ entry and no $i+1$ entry, $r^*(c)$ has a $3-i$ entry and no $4-i$ entry by Lemma~\ref{lemma:dualcounts}. In addition, for all $x$, we know that $(r([c,x]))^*=r^*([c,x])$, so by Lemma~\ref{lemma:dualcounts}, we can conclude that $d$ is the minimal column such that $r^*([c,d])$ has an equal number of $3-i$ and $4-i$ entries. This means that by definition, $c$ and $d$ are bonded in $r^*$. 

    Since $(r^*)^*=r$, this is sufficient to show that for $c < d$, $c$ and $d$ are bonded in $r$ iff they are bonded in $r^*$. By Proposition~\ref{prop:ballotcondition}, $r$ is ballot-admissible iff for all columns $d$ with an $i+1$ entry and no $i$ entry in $r$, there exists some $c < d$ with $c$ and $d$ bonded. All such columns $d$ correspond by Lemma~\ref{lemma:dualcounts} to columns in $r^*$ with a $4-i$ entry and no $3-i$ entry. Note that $i\mapsto 3-i$ is a bijection and this criterion for ballot-admissibility is expressed in terms of bonds between columns. This is sufficient to conclude that as bonds between columns of $r$ are preserved by $r\mapsto r^*$, so is ballot-admissibility.
\end{proof}

\begin{remark}
    For the remainder of this paper, since $r\mapsto r^*$ preserves both ballot-admissibility and bonds between columns, we will treat ribbons $r$ as interchangeable with their dual ribbons $r^*$.
\end{remark}

\begin{definition}
    Let $r$ be a ballot-admissible ribbon or dual ribbon of length $n$. We say that $r$ is atomic if for all $i\ge 2$, $r([i,n])$ is not ballot-admissible.
\end{definition}

\begin{proposition}\label{prop:atomiccrit}
    A ballot-admissible ribbon $r$ is atomic iff for all columns $c\ge 2$ in $r$, there exist some columns $a,b$ in $r$ where $a < c\le b$ and $a$ and $b$ are bonded. 
\end{proposition}

\begin{proof}
    Let $r$ be ballot-admissible and of length $n$. Then, $r$ is atomic iff $r([c,n])$ is not ballot-admissible for any $c\ge 2$.
    
    Now, let $c\ge 2$ be arbitrary. By Proposition~\ref{prop:ballotcondition}, $r([c,n])$ is ballot-admissible iff all columns $d\in [c,n]$ with an $i+1$ entry and no $i$ entry are bonded to some column $e < d$ in $r([c,n])$. By definition, this is equivalent to there existing some maximal column $e\in [c,d-1]$ such that $\mu_i(r([e,d]))=\mu_{i+1}(r([e,d]))$. Now, note that $e$ is the maximal column in $[c,d-1]$ where $\mu_i(r([e,d]))=\mu_{i+1}(r([e,d]))$ iff $e$ is the maximal such column in $[1,d-1]$. As such, $d$ is bonded to some $e < d$ in $r([c,n])$ iff $d$ is bonded to $e$ in $r$. From this, we see that $d$ is bonded to some $e < d$ in $r([c,n])$ iff $d$ is bonded to some $e\in [c,d-1]$ in $r$.

    Thus, we conclude that $r([c,n])$ is ballot-admissible iff for all columns $d\in [c,n]$ where $r(d)$ has an $i+1$ entry and no $i$ entry, $d$ is bonded in $r$ to some $e\in [c,d-1]$. Now, we note that in $r$, since $r$ is ballot-admissible, any such $d$ is bonded to some $e\in [1,d-1]$. As such, if $r([c,n])$ is not ballot-admissible iff there exists some $d\in [c,n]$ bonded in $r$ to some $e\in [1,c-1]$. 

    Now, we see that $r$ is atomic iff for all $c\ge 2$, $r([c,n])$ is not ballot-admissible. Equivalently, by our argument above, $r$ is atomic iff for all $c\ge 2$, there exist some $e,d$ such that $e$ and $d$ are bonded in $r$ and $e < c\le d$. This is our desired condition for $r$ to be atomic.
\end{proof}

\subsection{Peakless Motzkin Paths}
\begin{definition}
    Let $N=(1,1)$, $E=(1,0)$ and $S=(1,-1)$. We define a peakless Motzkin path of length $n$ as a lattice walk on $\mathbb{Z}^2$ with steps $N$, $E$, and $S$, starting at $(0,0)$, ending at $(n,0)$ and always staying at $y\ge 0$, with the pattern $NS$ forbidden.
\end{definition}

\begin{definition}
    For $n\ge 0$, let $a_n$ be the number of peakless Motzkin paths of length $n$. Note that we define that $a_0=1$.
\end{definition}

\begin{proposition}\label{prop:motzkinrecurrence}
    For all $n\ge 0$, $a_{n+1} = a_n + \sum_{j=1}^{n-1} a_j*a_{n-1-j}$.
\end{proposition}

\begin{proof}
    This is given by A004148 in~\cite{oeis}. Alternatively, we provide the following proof. Given some peakless Motzkin path $p$ of length $n+1$, either $p$ ends on $N$, $E$, or $S$. It is impossible that $p$ ends on $N$, as $p$ ends at $y=0$ and cannot go below $y=0$. Next, it is easy to see that peakless Motzkin paths of length $n+1$ ending on $E$ are in bijection with peakless Motzkin paths of length $n$ by removing the final $E$ step, so there are $a_n$ such paths.

    Finally, say that a peakless Motzkin path $p$ ends on $S$. Then, there must be at least one $N$ step in $p$. Since $p$ cannot go above $y=0$ without an $N$ step at $y=0$ existing, there must be an $N$ step which starts at $y=0$ in $p$. Take the last such step. Then, $p=p_1 N p_2 S$ where $p_2 S$ follows this final $N$ step. We note that since this is the final $N$ step at $y=0$, $p_2$ cannot ever reach $y=0$. Since $p_2$ starts at $y=1$ and ends at $y=1$, it is easy to see that $p_2$ is itself a peakless Motzkin path. Likewise, $p_1$ starts and ends at $y=0$ without going below $y=0$, so $p_1$ is also a peakless Motzkin path.

    We note that given any peakless Motzkin path $p$ which ends on $S$, we can uniquely define $p_1$ and $p_2$. We will count peakless Motzkin paths $p$ where $p_2$ has length $j$. Since $NS$ is forbidden, $p_2$ must have length $j\ge 1$. In addition, $j\le n-1$. Now, we the peakless Motzkin paths $p$ which have $p_2$ of some length $j$ correspond exactly to a choice of a peakless Motzkin path $p_1$ of length $n-1-j$ and a peakless Motzkin path $p_2$ of length $j$. As such, there exist $a_j * a_{n-1-j}$ such paths with $p_2$ of length $j$.

    We have now counted all peakless Motzkin paths of length $n+1$. This gives us
    \[a_{n+1} = a_n + \sum_{j=1}^{n-1} a_j * a_{n-1-j}\]

    as desired.
\end{proof}

\section{Constructing the Bijection}\label{sec:bijection}
\subsection{Algorithmic Steps}
\begin{definition}
    Let $r_0$ be a ballot-admissible ribbon. Let $r_0(c_0)$ be $i$-unbonded. Let $c_0'$ be the minimal column $ > c_0$ such that $r_0(c_0')$ is $i$-unbonded. We define $f(r_0,c_0)=(r_1,c_1)$ as follows. For all $c\ne c_0'$, $r_1(c)=r_0(c)$. We define $r_1(c_0')$ by changing the $i$ entry in $r_0(c_0')$ to $i+1$. Then, if $r_1(c_0')$ is unbonded, let $c_1=c_0'$. Otherwise, let $c_1$ be the largest column such that $r_0(c_1)$ is bonded but $r_0(c_1)$ is unbonded.
\end{definition}

\begin{remark}
    Note that $f$ is a partial function, as $f$ is only defined when there exists some column $d > c_0$ where $r_0(d)$ is $i$-unbonded.
\end{remark}

\begin{proposition}\label{prop:fwelldefined}
    Let $f(r_0,c_0)=(r_1,c_1)$. Then, one of the following holds:
    \begin{enumerate}
        \item[(i)] $r_1$ is a well-defined ballot-admissible ribbon and $c_1$ is a well-defined unbonded column in $r_1$.
        \item[(ii)] $r_0(c_0)=(1)$ or $(12)$ and $c_0'=c_0+1$.
    \end{enumerate}
\end{proposition}

\begin{proof}
    Let $r_0(c_0)$ be $i$-unbonded. We will show that if (ii) is not the case, then (i) holds. Suppose that (ii) is not the case.

    First, we show that $r_1$ is a Fibonacci ribbon or dual ribbon. For $c\ne c_0'$, $r_1(c)=r_0(c)$. For $c_0'$, if $r_0(c_0')=(i)$ then $r_0(c_0')=(i+1)$. Meanwhile, if $r_0(c_0')$ has all entries except for $i+1$ then $r_1(c_0')$ has all entries except for $i$, still in ascending order as we change the $i$ entry to an $i+1$ entry. Finally, we show that there are no $(12)(3)$ or $(1)(23)$ patterns in $r_1$. We know that $r_1(c_0')$ has an $i+1$ entry and no $i$ entry for some $i$, so $r_1(c_0')\ne (1)$ or $(12)$. 
    
    Now, consider the case where $r_1(c_0')=(3)$ and $r_0(c_0')=(2)$. Suppose that $r_1(c_0'-1)=(12)$. Then, $r_0(c_0'-1)=(12)$. Since (ii) is not the case by our assumption, $c_0 < c_0'-1 < c_0'$, and so $r_0(c_0'-1)=(12)$ must be bonded by minimality of $c_0'$. It then follows that $r_0(c_0'-1)$ is bonded to some $d > c_0$, so by Proposition~\ref{prop:naturalmatch}, $r_0(c_0')=(2)$ must also be bonded. This is impossible, so we conclude that $r_1(c_0'-1)\ne (12)$. An analogous argument holds in the case where $r_1(c_0')=(23)$. We may now conclude that $r_1$ has no $123$ patterns, so by Proposition~\ref{prop:ribboncondition}, $r_1$ is a Fibonacci ribbon or dual ribbon.

    Next, we show that $r_1$ is ballot-admissible. To do this, consider some arbitrary column $c\ge 1$. Let $r'=r_1([1,c])$. We want to show that for all $j$, $\mu_j(r')\ge \mu_{j+1}(r')$. If $c < c_0'$, then $r'=r_0([1,c])$, so since $r_0$ is ballot-admissible, we know that this is the case. Next, consider the case where $c \ge c_0$. Since $r_1$ is constructed by changing a single $i$ entry in $r_0$ to $i+1$, for $j\ne i$, 
    \[\mu_j(r')-\mu_{j+1}(r')\ge \mu_j(r_0([1,c]))-\mu_{j+1}(r_0([1,c]))\ge 0.\]

    Now, we know that for $j=i$, $c_0 < c_0'$ and $c_0' \le c$ are both $i$-unbonded. Then, by Corollary~\ref{cor:bondentryamounts}, $\mu_i(r_0([c_0,c_0'-1])) > \mu_{i+1}(r_0([c_0,c_0'-1]))$ and $\mu_i(r_0([c_0',c])) > \mu_{i+1}(r_0([c_0',c]))$. Also, since $r_0$ is ballot-admissible, $\mu_i(r_0([1,c_0-1]))\ge \mu_{i+1}(r_0([1,c_0-1]))$. Then, as $[1,c]=[1,c_0-1]+[c_0,c_0'-1]+[c_0',c]$, 
    \[\mu_i(r_0([1,c]))-\mu_{i+1}(r_0([1,c]))\ge 2\]

    And so, since we construct $r'$ by changing a single $i$ entry to an $i+1$ entry,
    \[\mu_i(r')-\mu_{i+1}(r')=\mu_i(r_0([1,c]))-\mu_{i+1}(r_0([1,c]))-2\ge 0\]

    as desired. Thus, we conclude that $r_1$ is ballot-admissible. Finally, we will show that $c_1$ is well-defined. If $r_1(c_0')$ is unbonded, then we know that $c_1=c_0'$ is well-defined. Next, consider the other case. Let $j=3-i$. We note that $j=i\pm 1$, so either there exists one more $j$ entry or one fewer $j+1$ entry in $r_1$ than $r_0$. In either case, it follows that there exists some column which is $j$-unbonded in $r_1$ but not $j$-unbonded in $r_0$. Since by our assumption, $r_1(c_0')$ is not $j$-unbonded, any such column which is $j$-unbonded in $r_1$ but not $r_0$ must be bonded in $r_0$. Thus, there exists some column $c$ such that $r_0(c)$ is bonded and $r_1(c)$ is unbonded. This is sufficient to conclude that $c_1$ is well-defined.
\end{proof}

\begin{remark}
    For the remainder of this section, we only say that $f$ is well-defined when $r_1$ is a ballot-admissible Fibonacci ribbon. When $r_1$ would not be semistandard, we do not define $f$.
\end{remark}

\begin{proposition}\label{prop:fremovedbond}
    Let $f(r_0,c_0)=(r_1,c_1)$. If $r_0(c)$ is bonded and $r_1(c)$ is unbonded, then $c=c_1$.
\end{proposition}

\begin{proof}
    Suppose that $r_0(c)$ is bonded and $r_1(c)$ is unbonded. Let $r_0(c_0)$ be $i$-unbonded and let $r_1(c)$ be $j$-unbonded. Since $r_0(c_0')$ is unbonded by definition, $c\ne c_0'$. In addition, let $c$ be bonded in $r_0$ to $c'$. Then, by Proposition~\ref{prop:bondchangecrit}, $c < c_0' \le c'$. By Proposition~\ref{prop:bondentryamounts}, $\mu_j(r_0([c,c']))=\mu_{j+1}(r_0([c,c']))$. However,  $c$ is $j$-unbonded in $r_1$. As a result, $\mu_j(r_1([c,c'])) > \mu_{j+1}(r_1([c,c']))$ by Corollary~\ref{cor:bondentryamounts}. It follows that $j=i\pm 1$, so $j\ne i$. In addition, since we know that $i,j\in \{1,2\}$, it follows that $j=3-i$. In this way, we see that any column $c$ where $r_0(c)$ is bonded and $r_1(c)$ is unbonded must be $j$-unbonded in $r_1$ for $j=3-i$.

    Note that if $c_1=c_0'$, then by definition, $r_1(c_1)$ is $j$-unbonded, and if $c_1\ne c_0'$ then $r_1(c_1)$ is also $j$-unbonded by the argument above.

    Now, suppose for the sake of contradiction that $c\ne c_1$. By definition of $f$, it follows that $c < c_1$. Furthermore, either $c_1 = c_0'$ or $r_0(c_1)$ is bonded while $r_1(c_1)$ is unbonded, so $c_1 \le c_0'\le c'$. Since $r_1(c_1)$ is $j$-unbonded, by Corollary~\ref{cor:bondentryamounts}, $\mu_j(r_1([c_1,c'])) > \mu_{j+1}(r_1([c_1,c']))$. In addition, since $c$ and $c'$ are bonded in $r_0$, $\mu_j(r_0([c,c']))=\mu_{j+1}(r_0([c,c']))$ by Proposition~\ref{prop:bondentryamounts}. Furthermore, as $j=i\pm 1$, we either add exactly one additional $j$ entry or remove one $j+1$ entry from $c_0'\in [c,c']$ when we construct $r_1$ from $r_0$. As such, $\mu_j(r_1([c,c']))=\mu_{j+1}(r_1([c,c']))+1$. In addition, since $c_1\in [c+1,c']$ is $j$-unbonded in $r_1$, we see that $\mu_j(r_1([c_1,c'])) > \mu_{j+1}(r_1([c_1,c']))$ by Corollary~\ref{cor:bondentryamounts}. As $[c,c']-[c_1,c']=[c,c_1-1]$, it follows that $\mu_j(r_1([c,c_1-1]))\le \mu_{j+1}(r_1([c,c_1-1]))$. By Corollary~\ref{cor:bondentryamounts}, we conclude that $c$ is bonded to some column $\le c_1-1$ in $r_1$, which is impossible, as $c$ is unbonded in $r_1$. As such, we have a contradiction. By contradiction, it must be the case that for $c$ such that $r_0(c)$ is bonded and $r_1(c)$ is unbonded, $c=c_1$.
\end{proof}

\begin{proposition}\label{prop:faddedbond}
    Let $f(r_0,c_0)=(r_1,c_1)$. Then, $c_0$ and $c_0'$ are bonded in $r_1$. Furthermore, for $c\ne c_0'$, if $r_0(c)$ is unbonded and $r_1(c)$ is bonded, then $c=c_0$.
\end{proposition}

\begin{proof}
    Let $r_0(c_0)$ be $i$-unbonded. First, we show that $c_0$ and $c_0'$ are bonded in $r_1$. Let $d$ be the unique column such that $d < c_0'$ and $d$ and $c_0'$ are bonded in $r_1$. By definition of $f$, $r_1(c_0')$ has an $i+1$ entry and no $i$ entry, so $r_1(d)$ must have an $i$ entry and no $i+1$ entry.

    First, if $d < c_0$, then by Proposition~\ref{prop:naturalmatch}, $c_0$ is bonded to some column $< c_0'$ in $r_1$. By Proposition~\ref{prop:bondchangecrit}, $c_0$ is likewise bonded in $r_0$, which is impossible. Meanwhile, if $d > c_0$, then since $c_0'$ is $i$-unbonded in $r_0$, by Proposition~\ref{prop:naturalmatch}, either $d$ is unbonded or $d$ is bonded to some column $<c_0'$ in $r_0$. By definition of $f$, since $c_0 < d < c_0'$, $d$ cannot be $i$-unbonded. In addition, if $d$ is bonded to some column $<c_0'$ in $r_0$, the same is true in $r_1$ by Proposition~\ref{prop:bondchangecrit}, which is impossible, as $d$ is bonded to $c_0'$ in $r_1$ by definition. So, we obtain a contradiction if $d < c_0$ or $d > c_0$. We conclude that $d=c_0$. Thus, $c_0$ and $c_0'$ are bonded in $r_1$.

    Next, say that for $c\ne c_0'$, $r_0(c)$ is $j$-unbonded and $r_1(c)$ is bonded. We want to show that $c=c_0$. Let $c$ be bonded to $c'$ in $r_1$. Then, $\mu_j(r_1([c,c'])) = \mu_{j+1}(r_1([c,c']))$ by Proposition~\ref{prop:bondentryamounts}. Meanwhile, since $c$ is $j$-unbonded in $r_0$, $\mu_j(r_0([c,c'])) > \mu_{j+1}(r_0([c,c']))$ by Corollary~\ref{cor:bondentryamounts}. It then follows that $j=i$.

    Finally, suppose for the sake of contradiction that $c\ne c_0$. Then, $c'\ne c_0'$ since $c_0$ and $c_0'$ are bonded in $r_1$. Now, by Proposition~\ref{prop:bondchangecrit}, $c < c_0' < c'$. Note that $c'$ has an $i+1$ entry and no $i$ entry, so there must be some $e < c'$ where $e$ and $c'$ are bonded in $r_0$. Since $c_0' < c'$ and $r_0(c_0')$ is $i$-unbonded by definition, $c_0' < e$ by Proposition~\ref{prop:naturalmatch}. Then, since $e$ and $c'$ are bonded in $r_0$ and $e,c' > c_0'$, $e$ and $c'$ are also bonded in $r_1$ by Proposition~\ref{prop:bondchangecrit}. It follows that $c=e$, and that $r_0(c)$ is not $i$-unbonded. This is impossible by definition of $c$, so by contradiction, $c=c_0$.
\end{proof}

\begin{definition}
    Let $r_0$ be a ballot-admissible ribbon. Let $r_0(c_0)$ be $i$-unbonded. Let $c_0' \ge c_0$ be the largest column bonded to some column preceding $c_0$. We define $g(r_0,c_0)=(r_1,c_1)$ as follows. Let $r_1(c)=r_0(c)$ for all $c\ne c_0'$. Let $r_1(c)$ be constructed by changing the $4-i$ entry in $r_0(c)$ to $3-i$. Then, let $c_1$ be the largest column such that $c_1\ne c_0'$, $r_0(c_1)$ is bonded, and $r_1(c_1)$ is unbonded.
\end{definition}

\begin{remark}
    As with $f$, $g$ is not always well-defined, since there may be no column $\ge c_0$ in $r_0$ bonded to a column preceding $c_0$. 
\end{remark}

\begin{proposition}\label{prop:gwelldefined}
    Let $g(r_0,c_0)=(r_1,c_1)$. Then, $r_1$ is a well-defined ballot-admissible ribbon and $c_1$ is a well-defined unbonded column in $r_1$.
\end{proposition}

\begin{proof}
    Let $r_0(c_0)$ be $i$-unbonded. We will show that $r_1$ is ballot-admissible and that $c_1$ is well-defined.
    
    First, we know that some column $d < c_0$ is bonded to $c_0' \ge c_0$ in $r_0$. Since $r_0(c_0)$ is $i$-unbonded, it follows by Proposition~\ref{prop:naturalmatch} that the bond between $d$ and $c_0'$ must be between a $j$ entry in $d$ and a $j+1$ entry in $c_0'$ where $j\ne i$. Since $i,j\in \{1,2\}$, we conclude that $j=3-i$. This means that $r_0(c_0')$ has a $4-i$ entry and no $3-i$ entry. As a result, the entries of $r_1$ are well-defined. Moreover, since $r_0(c_0')$ has no $3-i$ entry, it is easy to see that $r_1(c_0')$ still has strictly increasing entries from top to bottom. By Proposition~\ref{prop:ribboncondition}, to show that $r_1$ is a Fibonacci ribbon or dual ribbon, it is sufficient to show that there are no $(12)(3)$ or $(12)(3)$ patterns in $r_1$. 

    To avoid these patterns, it is easy to see that we need only consider two cases. First, suppose that $r_0(c_0')=(13)$ and $r_0(c_0'+1)=(3)$. Then, we obtain a $(12)(3)$ pattern in $r_1$, which is forbidden. However, this is impossible, as by Proposition~\ref{prop:naturalmatch}, if $r_0(c_0')=(13)$ is bonded to a column preceding $c_0$, the same must be true for $r_0(c_0'+1)$. This is impossible by the maximality of $c_0'$. By the same argument, the case where $r_0(c_0')=(2)$ and $r_0(c_0'+1)(23)$ is impossible, so we cannot obtain a $(1)(23)$ pattern in $r_1$. Thus, $r_1$ is a well-defined Fibonacci ribbon.

    Next, we show that $r_1$ is ballot-admissible. Let $c\ge 1$ be some arbitrary column. Let $r'=r_1([1,c])$. We want to show that if for all $j$, $\mu_j(r')\ge \mu_{j+1}(r')$. For all $c < c_0'$, $r'=r_0([1,c])$, so this holds due to the ballot-admissibility of $r_0$. Now, consider the case where $c\ge c_0'$. It is easy to see that $\mu_j(r')\ge \mu_{j+1}(r')$ must be the case unless $j=4-i$ or $j+1=3-i$. Then, $j\in\{1,2\}\cap \{2-i,4-i\}$. Since $i\in\{1,2\}$, it is easy to see that this uniquely determines $j$. If $i=1$ then $j=1$ and if $i=2$ then $j=2$, so $i=j$ is the only case which we need check to conclude that $r_1$ is ballot-admissible.
    
    Since $c_0\in [1,c]$ is $i$-unbonded in $r_0$, $\mu_i(r_0([c_0,c])) > \mu_{i+1}(r_0([c_0,c]))$ by Corollary~\ref{cor:bondentryamounts}. In addition, since $r_0$ is ballot-admissible, $\mu_i(r_0([1,c_0-1]))\ge \mu_{i+1}(r_0([1,c_0-1]))$. So, we conclude that
    \[\mu_i(r_0([1,c]))\ge \mu_{i+1}(r_0([1,c]))+1\]

    And it easily follows that
    \[\mu_i(r_1([1,c]))-\mu_{i+1}(r_1([1,c])) = \mu_i(r_0([1,c])) - \mu_{i+1}(r_0([1,c])) - 1\ge 0\]

    which is sufficient to show that $r_1$ is ballot-admissible as desired. Finally, $r_1$ has one more $3-i$ entry and one fewer $4-i$ entry than exist in $r_0$. As such, there are at least two additional $(3-i)$-unbonded columns in $r_1$ which were not $(3-i)$-unbonded in $r_0$. This means that some such column $c\ne c_0'$ must exist, and since $c\ne c_0'$, and $c$ is not $(3-i)$-unbonded in $r_0$, it follows that $r_0(c)$ is bonded. We may now conclude that $c_1$ is well-defined.
\end{proof}

\begin{proposition}\label{prop:gremovedbond}
    Let $g(r_0,c_0)=(r_1,c_1)$. Then, $r_1(c_0')$ is unbonded and if $c\ne c_0'$, $r_0(c)$ is bonded, and $r_1(c)$ is unbonded, then $c=c_1$.
\end{proposition}

\begin{proof}
    Let $r_0(c_0)$ be $i$-unbonded. First, we show that $r_1(c_0')$ is unbonded. Suppose for the sake of contradiction that $r_1(c_0')$ is bonded to some $d>c_0'$. We know that $r_1(c_0')$ has a $3-i$ entry and no $4-i$ entry. As such, if $c_0'$ is bonded to some column $d > c_0'$, then $d$ has a $4-i$ entry and no $3-i$ entry. Since $d\ne c_0'$, $r_0(d)=r_1(d)$, and so $d$ must be bonded to some $e < d$ in $r_0$. Note that $e\ne c_0'$, since $r_0(c_0')$ does not have a $3-i$ entry. By Proposition~\ref{prop:bondchangecrit}, since $e$ and $d$ are not bonded in $r_1$, $e < c_0'$. Now, we know that in $r_0$, $d > c_0'$ is bonded to some $e < c_0'$. In addition, $r_0(c_0')$ also has a $4-i$ entry and no $3-i$ entry. So, by Proposition~\ref{prop:naturalmatch}, since $c_0'$ is bonded to a column preceding $c_0$ in $r_0$, the same must be the case for $d$. However, $d > c_0'$, so this is impossible by maximality of $c_0'$. Thus, by contradiction, we find that $r_1(c_0')$ is unbonded as desired.

    Next, let $c\ne c_0'$ be some column where $r_0(c)$ is bonded and $r_1(c)$ is $j$-unbonded. Then, there exists some column $d$ such that $c$ and $d$ are bonded in $r_0$. By Proposition~\ref{prop:bondentryamounts}, $\mu_j(r_0([c,d]))=\mu_{j+1}(r_0([c,d]))$. However, since $c$ is $j$-unbonded in $r_1$, $\mu_j(r_1([c,d])) > \mu_{j+1}(r_1([c,d]))$. It follows that $j=3-i$ since this is the only $j$ for which we have an additional $j$ entry or one fewer $j+1$ entry in $r_1$.

    Finally, suppose for the sake of contradiction that $c\ne c_1$. It follows that $c < c_1$. By Proposition~\ref{prop:bondchangecrit}, since both $c$ and $c_1$ are bonded in $r_0$ but unbonded in $r_1$, $c,c_1 < c_0'$, and $d \ge c_0'$. Now, we know that $c$, $c_1$, and $c_0'$ are all $j$-unbonded in $r_1$ for $j=3-i$. So, $\mu_{j}(r_1([c,c_1-1])) > \mu_{j+1}(r_1([c,c_1-1]))$. In addition, $\mu_{j}(r_1([c_1,c_0'-1])) > \mu_{j+1}(r_1([c_1,c_0'-1]))$ and $\mu_{j}(r_1([c_0',d])) > \mu_{j+1}(r_1([c_0',d]))$. We may then conclude that
    \[\mu_j(r_1([c,d]))-\mu_{j+1}(r_0[c,d])\ge 3\]

    And so
    \[\mu_j(r_0([c,d]))-\mu_{j+1}(r_0([c,d]))=\mu_j(r_1([c,d]))-\mu_{j+1}(r_0[c,d])-2\ge 1.\]

    However, this is impossible by Proposition~\ref{prop:bondentryamounts}, as $c$ and $d$ are bonded in $r_0$. As such, we have a contradiction, and by contradiction, $c=c_1$ as desired.
\end{proof}

\begin{proposition}\label{prop:gaddedbond}
    Let $g(r_0,c_0)=(r_1,c_1)$. Let $r_0(c_0)$ be $i$-unbonded. Suppose that $c_0 < c_0'$ and for all $c_0 < c\le c_0'$, $r_0(c)$ is not $i$-unbonded. Then, $c_0$ is bonded in $r_1$.
\end{proposition}

\begin{proof}
    Observe that when we change a $4-i$ entry to a $3-i$ entry, this gives us either an additional $i+1$ entry or one fewer $i$ entry in $r_1$ than in $r_0$, depending on if $i=1$ or $i=2$.

    First, we consider the case where for all $c_0 < c \le c_0'$, $r_0(c)$ does not have an $i$ entry with no $i+1$ entry. Then, it follows that $\mu_i(r_0([c_0+1,c_0']))\le \mu_{i+1}(r_0([c_0+1,c_0']))$. In addition, we know that $c_0$ is $i$-unbonded in $r_0$, so $\mu_i(r_0([c_0,c_0'])) > \mu_{i+1}(r_0([c_0,c_0']))$. It follows that $\mu_i(r_0([c_0,c_0']))=\mu_{i+1}(r_0([c_0,c_0']))+1$. Then, since $r_1$ has either one additional $i+1$ or one fewer $i$ entry, $\mu_i(r_1([c_0,c_0']))=\mu_{i+1}(r_1([c_0,c_0']))$. So, we see that by Proposition~\ref{prop:bondentryamounts}, $r_1(c_0)$ is not $i$-unbonded, and instead is bonded as desired.

    Next, consider the case where there exists some column $c_0 < c\le c_0'$ with an $i$ entry but no $i+1$ entry in $r_0$. By our assumption, any such $c$ must be bonded to a later column, as it would otherwise be $i$-unbonded. Let $d$ be the maximal column which is bonded in $r_0$ to some $c$ where $c_0 < c\le c_0'$ and $r_0(c)$ has an $i$ entry but no $i+1$ entry. Note that $r_0(c_0')$ is bonded to some column preceding $c_0$ while $c_0$ is $i$-unbonded, so $r_0(c_0')$ cannot have an $i+1$ entry and no $i$ entry. Meanwhile, $r_0(d)$ must have an $i+1$ entry and no $i$ entry, so $d\ne c_0'$. 

    Consider the case where $d < c_0'$. By our assumption, any $i$ entry in $r_0([c_0+1,d])$ must be bonded. Furthermore, by definition of $d$, any $i$ entry in $r_0([c_0+1,d])$ must be bonded to an $i+1$ entry in $r_0([c_0+1,d])$. This means that for every $i$ entry in $r_0([c_0+1,d])$ there exists a unique $i+1$ entry in $r_0([c_0+1,d])$. As such, we find that $\mu_i(r_0([c_0+1,d])) \le \mu_{i+1}(r_0([c_0+1,d]))$. Furthermore, say that for some $e\in [d+1,c_0']$, $r_0(e)$ has an $i$ entry with no $i+1$ entry. Then, by our assumption, $e$ must be bonded to a later column, which would then be a greater column than $d$. By maximality of $d$, this is impossible, so no columns in $r_0([d+1,c_0'])$ have an $i$ entry with no $i+1$ entry. As a result,
    \[\mu_i(r_0([c_0+1,c_0']))\le \mu_{i+1}(r_0([c_0+1,c_0']))\]

    And it follows that since $r_0(c_0)$ has an $i$ entry and no $i+1$ entry,
    \[\mu_i(r_0([c_0,c_0']))\le \mu_{i+1}(r_0([c_0,c_0']))+1.\]
    
    We note that in $r_1$, there is either an additional $i+1$ entry or one fewer $i$ entry in $c_0'$, so 
    \[\mu_i(r_1([c_0,c_0']))\le \mu_{i+1}(r_1([c_0,c_0']))\]

    which means that $c_0$ is bonded in $r_1$ by Corollary~\ref{cor:bondentryamounts} as desired. Finally, we will consider the case where $d > c_0'$. Let $e\in [c_0+1, c-1]$ where $c$ is bonded to $d$ in $r_0$, and say that $r_0(e)$ has an $i$ entry and no $i+1$ entry. By our assumption, $e$ must be bonded in $r_0$. By Proposition~\ref{prop:naturalmatch}, $e$ is either bonded to some column $\le c-1$ or $>d$. The latter is impossible by maximality of $d$, so we see that $e$ must be bonded to some column $\le c-1$. This means that for any $i$ entry in $r_0([c_0+1,c-1])$, there exists a unique $i+1$ entry in $r_0([c_0+1,c-1])$. So, $\mu_i(r_0([c_0+1,c-1]))\le \mu_{i+1}(r_0([c_0+1,c-1]))$. In addition, since $c$ and $d$ are bonded in $r_0$, $\mu_i(r_0([c,d]))=\mu_{i+1}(r_0([c,d]))$. Then,
    \[\mu_i(r_0([c_0+1,d]))\le \mu_{i+1}(r_0([c_0+1,d]))\]

    And so it follows that
    \[\mu_i(r_0([c_0,d]))\le \mu_{i+1}(r_0([c_0,d]))+1.\]

    Since $c_0'\in [c_0,d]$ and we know that $c_0'$ has either one more $i+1$ or one fewer $i$ in $r_1$ than $r_0$, we conclude that
    \[\mu_i(r_1([c_0,d]))\le \mu_{i+1}(r_1([c_0,d]))\]

    And so $c_0$ is bonded in $r_1$ as desired. We have now shown that $r_1(c_0)$ is bonded in all cases.
\end{proof}

\begin{remark}
    Both $f$ and $g$ as defined in this section can be described with the coplactic operations on tableaux given in~\cite[Section 3.1]{leeuwen2001}.
\end{remark}

\subsection{Useful Properties}
In this subsection, we give various properties of $f$ and $g$ which we will find useful later.

\begin{proposition}\label{prop:fterminate}
    Let $f(r_0,c_0)=(r_1,c_1)$ and $f(r_1,c_1)=(r_2,c_2)$. Let $c_0$ be bonded to $c_0'$ in $r_1$. Let $c_1$ be bonded to $c_1'$ in $r_2$. Then, $c_0' < c_1'$.
\end{proposition}

\begin{proof}
    Let $r_0(c_0)$ be $i$-unbonded. In the case where $r_1(c_0')$ is unbonded, $c_1=c_0'$. We know that $c_1 < c_1'$ by definition, so in this case, $c_0' < c_1'$.

    Next, consider the case where $c_0'$ is not unbonded. Then, by Proposition~\ref{prop:fremovedbond}, $c_1$ is the unique column such that $r_0(c_1)$ is bonded and $r_1(c_1)$ is unbonded. Note that $r_1(c_1)$ is $j$-unbonded for $j=3-i$. By Proposition~\ref{prop:bondchangecrit}, since $c_1$ is bonded in $r_0$ and unbonded in $r_1$, $r_0(c_1)$ is bonded to some column $\ge c_0'$. Then, by Proposition~\ref{prop:naturalmatch}, there exists no $j$-unbonded column $r_0(c)$ where $c\in [c_1,c_0']$. 

    Finally, suppose that $c_1' \le c_0'$. Then, $c_1' > c_1$ by definition, and $r_1(c_0')$ is not unbonded in this case, so $c_1'\in [c_1+1,c_0'-1]$. Since $c_1$ is the only column where $r_0(c_1)$ is bonded and $r_1(c_1)$ is unbonded, and $r_1(c_1')$ is $j$-unbonded by definition, $r_0(c_1')$ cannot be bonded. Since $c_1'\ne c_0'$, $r_0(c_1')=r_1(c_1')$ and so it follows that $r_0(c_1')$ is also $j$-unbonded. However, this is impossible by our argument above since $c_1'\in [c_1,c_0']$. Thus, we have a contradiction, and we conclude that $c_1' > c_0'$. 
\end{proof}

\begin{proposition}\label{prop:gterminate}
    Let $g(r_0,c_0)=(r_1,c_1)$. Then, $c_1 < c_0$.
\end{proposition}

\begin{proof}
    Let $r_0(c_0)$ be $i$-unbonded. Let $j=3-i$. We note that $r_1(c_1)$ is $j$-unbonded. Let $c < c_0$ be bonded to $c_0'$ in $r_0$. Since $r_0(c_0)$ is $i$-unbonded, $r_0(c)$ has a $j$ entry and no $j+1$ entry. Suppose for the sake of contradiction that $c_0\le c_1$. Then, we know that $c < c_1$. Since $r_0(c_1)$ is bonded and $r_1(c_1)$ is $j$-unbonded, we know that $c_1 < c_0'$ by Proposition~\ref{prop:bondchangecrit}. Furthermore, by Proposition~\ref{prop:gremovedbond}, $c_1$ is the only column such that $c_1\ne c_0'$, $r_0(c_1)$ is bonded, and $r_1(c_1)$ is unbonded. Since $c\ne c_0'$, $c\ne c_1$, and $r_0(c)$ is bonded, it follows that $r_1(c)$ cannot be unbonded. In addition, $r_1(c)=r_0(c)$, so $r_1(c)$ has a $j$ entry and no $j+1$ entry, so $c$ must be bonded to some column $>c$ in $r_1$. Then, by Proposition~\ref{prop:bondchangecrit}, since $c$ is bonded to $c_0'$ in $r_0$, it cannot be the case that $c$ is bonded to some column $<c_0'$ in $r_1$. We conclude that $r_1(c)$ is bonded to some column $\ge c_0'$. Since $c < c_1 < c_0'$, by Proposition~\ref{prop:naturalmatch}, $r_1(c_1)$ cannot be $j$-unbonded. As such, we have a contradiction, and we conclude that $c_1 < c_0$. 
\end{proof}

\begin{proposition}\label{prop:fuseful}
    Let $f(r_0,c_0)=(r_1,c_1)$. Suppose that for columns $c < c_1$ and $d > c_0'$, $c$ and $d$ are bonded in $r_1$. Then, $c < c_0$.
\end{proposition}

\begin{proof}
    Let $r_0(c_0)$ be $i$-unbonded. Then, we know that $r_1(c_1)$ is $j$-unbonded for $j=3-i$. In addition, $c < c_1$ is bonded to $d > c_0' \ge c_1$ in $r_1$, while $r_1(c_1)$ is $j$-unbonded. It follows that $c$ cannot have a $j$ entry and no $j+1$ entry. As such, $c$ has an $i$ entry and no $i+1$ entry.

    Now, in $r_1$, we know that $c$ is bonded to $d > c_0'$, while $c_0$ is bonded to $c_0'$ by Proposition~\ref{prop:faddedbond}. In addition, $c < c_1 \le c_0'$, so by Proposition~\ref{prop:naturalmatch}, $c < c_0$ as desired.
\end{proof}

\begin{proposition}\label{prop:fpreservation}
    Let $f(r_0,c_0)=(r_1,c_1)$. Let $c < c_0$ and $d > c_0'$. If $c$ is bonded to $d$ in $r_1$, then $c$ is bonded to a column $\ge c_0'$ in $r_0$.
\end{proposition}

\begin{proof}
    By Proposition~\ref{prop:bondchangecrit}, since $c < c_0 < c_0'$ is bonded to $d > c_0'$ in $r_1$, either $c$ is unbonded in $r_0$, or $c$ is bonded to a column $\ge c_0'$ in $r_0$. If $c$ is unbonded in $r_0$, as it is bonded in $r_1$, by Proposition~\ref{prop:faddedbond}, either $c=c_0$ or $c=c_0'$. However, $c < c_0 < c_0'$, so this is impossible. As such, $c$ must be bonded to some column $\ge c_0'$ in $r_0$.
\end{proof}

\begin{proposition}\label{prop:ginverse}
    Let $g(r_0,c_0)=(r_1,c_1)$. If $r_1(c_0)$ is bonded or $c_0=c_0'$, $f(r_1,c_1)=(r_0,c_0)$.
\end{proposition}

\begin{proof}
    Let $r_0(c_0)$ be $i$-unbonded. By definition, $c_0'$ is the largest column bonded to some column $d < c_0$ in $r_0$. In addition, note that $c_1 < c_0'$, and $c_1$ and $c_0'$ are both $j$-unbonded in $r_1$ for $j=3-i$ by Proposition~\ref{prop:gremovedbond}.

    Next, we show that $c_0'$ is the minimal $j$-unbonded column after $c_1$ in $r_1$. Suppose for the sake of contradiction that there exists some $j$-unbonded column $r_1(c)$ where $c_1 < c < c_0'$. We know that $r_0(c_1)$ is bonded to some $d\ge c_0'$. Since $c_1 < c < d$, by Proposition~\ref{prop:naturalmatch}, $r_0(c)$ cannot be $j$-unbonded, and it follows that $r_0(c)$ is bonded. In addition, $c\ne c_1$ and $c\ne c_0'$, so by Proposition~\ref{prop:gremovedbond}, it cannot be the case that $r_0(c)$ is bonded and $r_1(c)$ is unbonded. However, by definition, $r_1(c)$ is unbonded and we showed that $r_0(c)$ is bonded. As such, we have a contradiction, and $c_0'$ must be the minimal $j$-unbonded column such that $c_0' > c_1$.

    Finally, since $c_0'$ is the minimal $j$-unbonded column after $c_1$ in $r_1$, by definition of $f$, the $j$ entry in $r_1(c_0')$ is changed to $j+1$. Observe that by the definition of $g$, this constructs $r_0(c_0')$. Furthermore, by our assumption, either $r_1(c_0)$ is bonded or $c_0=c_0'$. We know that $r_0(c_0)$ is unbonded by definition, so if $r_1(c_0)$ is bonded, by Proposition~\ref{prop:fremovedbond}, it follows that $f(r_1,c_1)=(r_0,c_0)$. Meanwhile, if $c_0=c_0'$, then since $f$ changes the entries of $c_0=c_0'$, and $r_0(c_0)$ is unbonded, $f(r_1,c_1)=(r_0,c_0)$ by definition.
\end{proof}

\begin{proposition}\label{prop:gpreservation}
    Let $g(r_0,c_0)=(r_1,c_1)$. Let $c < c_1$ be bonded to $d > c$ in $r_0$. Then, $c$ and $d$ are also bonded in $r_1$. 
\end{proposition}

\begin{proof}
    By Proposition~\ref{prop:gterminate}, $c_1 < c_0$. It follows that $c < c_0$. Now, we know that $c$ is bonded to some $d > c$. Since $c_0'$ is the maximal column bonded to some column $ < c_0$, it follows that $d\le c_0'$.

    Now, suppose for the sake of contradiction that $d = c_0'$. Let $r_0(c_0)$ be $i$-unbonded. Let $j=3-i$. Since $c < c_0$ and $d \ge c_0$, it must be the case that $c$ has a $j$ entry and no $j+1$ entry by Proposition~\ref{prop:naturalmatch}. Likewise, we know that $c_1$ has a $j$ entry and no $j+1$ entry. Since $c < c_1$ and $c_1 < d$, it follows that $c_1$ is bonded in $r_0$ to some column $< d=c_0'$ by Proposition~\ref{prop:naturalmatch}. Then, by Proposition~\ref{prop:bondchangecrit}, $r_1(c_1)$ is also bonded, which is impossible. Thus, we have a contradiction and see that $d\ne c_0'$.

    We conclude that $d < c_0'$. Since $c,d < c_0'$ and $c$ and $d$ are bonded in $r_0$, they remain bonded in $r_1$ by Proposition~\ref{prop:bondchangecrit}.
\end{proof}

\begin{proposition}\label{prop:gdouble}
    Let $g(r_0,c_0)=(r_1,c_1)$. Suppose $r_1(c_0)$ is bonded or $c_0=c_0'$, and there exists some column $\ge c_1$ bonded in $r_1$ to a column preceding $c_1$. Then, $g(r_1,c_1)=(r_2,c_2)$ is well-defined, and $r_2(c_1)$ is bonded or $c_1=c_1'$.
\end{proposition}

\begin{proof}
    Let $r_0(c_0)$ be $i$-unbonded. We assume that there exists some column $\ge c_1$ bonded in $r_1$ to a column preceding $c_1$. As such, $g(r_1,c_1)=(r_2,c_2)$ is well-defined by definition of $g$. 
    We now show that either $r_2(c_1)$ is bonded or $c_1=c_1'$. We know that $r_1(c_1)$ is $j$-unbonded for $j=3-i$. By Proposition~\ref{prop:gaddedbond}, it is sufficient to show that for all $c_1 < c\le c_1'$, $r_1(c)$ is not $j$-unbonded. 

    First, by definition, there exists some $d < c_1$ where $d$ and $c_1'$ are bonded in $r_1$. Since $r_1(c_1)$ is $j$-unbonded, it follows that $d$ has an $i$ entry and no $i+1$ entry. Then, by Proposition~\ref{prop:gterminate}, $c_1 < c_0$, so $d < c_0$. Since $r_0(c_0)$ is $i$-unbonded, either $r_0(d)$ is $i$-unbonded or $d$ is bonded to some column $<c_0$ in $r_0$. Since we assume that either $r_1(c_0)$ is bonded or $c_0=c_0'$, $f(r_1,c_1)=(r_0,c_0)$ by Proposition~\ref{prop:ginverse}. We know that $r_1(d)$ is bonded, so if $r_0(d)$ is unbonded, then $d=c_0$ by Proposition~\ref{prop:fremovedbond}. However, we already showed that $d < c_0$, so this is impossible. As such, we conclude that $d$ is bonded to some column $e <c_0$ in $r_0$. Then, by definition, $c_0 \le c_0'$, so $d,e < c_0'$, and by Proposition~\ref{prop:bondchangecrit}, $d$ and $e$ are also bonded in $r_1$. It follows that $e=c_1'$, so $c_1' < c_0$. 

    Now, since $c_1' < c_0$, it is sufficient to show that no columns $r_1(c)$ for $c_1 < c < c_0$ which are $j$-unbonded may exist. Let $c_1 < c < c_0$. We know that $r_0(c_1)$ has a $j$ entry and no $j+1$ entry. Furthermore, since $r_0(c_1)$ is bonded while $r_1(c_1)$ is unbonded, $r_0(c_1)$ is bonded to some $e\ge c_0'$ by Proposition~\ref{prop:bondchangecrit}. Note that $c_1 < c < e$. Then, by Proposition~\ref{prop:naturalmatch}, since $c_1$ and $e$ are bonded in $r_0$, $r_0(c)$ cannot be $j$-unbonded. Since $c\ne c_0'$ and $c\ne c_1$, by Proposition~\ref{prop:gremovedbond}, it cannot be the case that $r_0(c)$ is bonded and $r_1(c)$ is unbonded. If $r_1(c)$ is $j$-unbonded, since $c\ne c_0'$, it then follows that $r_0(c)$ is $j$-unbonded or $r_0(c)$ is bonded. We have shown above that $r_0(c)$ cannot be $j$-unbonded, and it cannot be the case that $r_0(c)$ is bonded while $r_1(c)$ is unbonded. As such, $r_1(c)$ cannot be $j$-unbonded. We now conclude that $r_2(c_1)$ is bonded or $c_1=c_1'$ by Proposition~\ref{prop:gaddedbond}.
\end{proof}

\begin{proposition}\label{prop:fdouble}
    Let $f(r_0,c_0)=(r_1,c_1)$. Let $r_1(c_1)$ be $i$-unbonded. Suppose that there exists some $i$-unbonded column after $c_1$ in $r_1$. Then, $f(r_1,c_1)$ is well-defined.
\end{proposition}

\begin{proof}
    By our assumption, we know that $c_1'$ is well-defined. Then, by Proposition~\ref{prop:fwelldefined}, it is sufficient to show that it cannot be the case that $r_1(c_1)=(1)$ or $(12)$ and $c_1'=c_1+1$. Suppose for the sake of contradiction that $r_1(c_1)=(1)$ or $(12)$ and $c_1'=c_1+1$.

    In the case where $c_1=c_0'$, we know that $c_0$ is bonded to $c_0'$ in $r_1$, so it cannot be the case that $r_1(c_1)=(1)$ or $(12)$. As such, we have a contradiction.
    
    Next, consider the case where $c_1 \ne c_0'$. We then know by definition that $r_0(c_1)$ is bonded while $r_1(c_1)$ is unbonded. By Proposition~\ref{prop:bondchangecrit}, it follows that $c_1 < c_0'$ and $c_1$ is bonded in $r_0$ to some column $\ge c_0'$. In addition, in this case, $r_1(c_0')$ is not unbonded, so $c_1'\ne c_0'$. As $c_1'=c_1+1$, it then follows that $c_1' < c_0'$. Then, since $r_0(c_1)$ has an $i$ entry and no $i+1$ entry, and $c_1$ is bonded in $r_0$ to some column $\ge c_0' > c_1'$, it cannot be the case that $r_0(c_1')$ is $i$-unbonded. Since $r_1(c_1')$ is $i$-unbonded, it follows that $r_0(c_1')$ is bonded while $r_1(c_1')$ is unbonded. As such, by Proposition~\ref{prop:fremovedbond}, $c_1'=c_1$, but $c_1'=c_1+1$, so this is impossible.

    Thus, we have a contradiction in both cases. We conclude that it cannot be the case that $r_1(c_1)=(1)$ or $(12)$ and $c_1'=c_1+1$. By Proposition~\ref{prop:fwelldefined}, we now know that $f(r_1,c_1)$ is well-defined.
\end{proof}

\subsection{Peakless Motzkin Path Bijection}
\begin{definition}
    Let $r$ be a ballot-admissible ribbon of length $n$. We define $\psi_1(r)$ to be an atomic ribbon of length $n+2$ as follows. Add a column $(12)$ as the rightmost column of $r$ in order to create $r_0$, a ballot-admissible ribbon of length $n+1$. Then, let $c_0$ be this new initial column. While $f(r_i,c_i)$ is well-defined, we let $(r_{i+1},c_{i+1})=f(r_i,c_i)$. If $f(r_i,c_i)$ is not well-defined, we add a final column as a leftmost column to the end of $r_i$ in order to construct $\psi_1(r)$. If $r_i(c_i)=(1)$ or $(13)$ then we add $(2)$ or $(23)$. If $r_i(c_i)=(12)$ or $(2)$ then we add $(13)$ or $(3)$. Note that this depends on the value of $n$ mod 2.
\end{definition}

\begin{proposition}
    Let $r$ be a ballot-admissible ribbon of length $n$ for $n\ge 1$. Then, $\psi_1(r)$ is a well-defined atomic ribbon of length $n+2$.
\end{proposition}

\begin{proof}
    Given $(r_{i+1},c_{i+1})=f(r_i,c_i)$, let $c_i'$ be the column whose entries in $r_i$ and $r_{i+1}$ differ.

    First, since $r$ is ballot-admissible, it immediately follows that $r_0=(12)r$ is ballot-admissible and that given $c_0=1$, $r_0(c_0)$ is unbonded. Now, we note that $r_0(c_0)=(12)$, while $r_0(c_0+1)=r(1)=(1)$, so it is impossible that $c_0'=c_0+1$. It follows that if $c_0'$ is well-defined, $f(r_0,c_0)$ is well-defined by Proposition~\ref{prop:fwelldefined}. Likewise, for $i\ge 1$, if $(r_i,c_i)=f(r_{i-1},c_{i-1})$ is well-defined, and $c_i'$ is well-defined, then by Proposition~\ref{prop:fdouble}, $f(r_i,c_i)$ is also well-defined.

    Now, by Proposition~\ref{prop:fterminate}, $c_i' < c_{i+1}'$ for all $i\ge 0$, so as we have only a finite number of columns, there eventually exists some $k\ge 0$ such that $f(r_k,c_k)$ is not well-defined. By our argument above, if $c_k'$ is well-defined, then $f(r_k,c_k)$ is well-defined. As such, since $f(r_k,c_k)$ is not well-defined, $c_k'$ cannot be well-defined. This means that if $c_k$ is $a$-unbonded, $c_k$ is the maximal $a$-unbonded column in $r_k$. Then, we construct $\psi_1(r)$ by adding a column to the end of $r_k$ with an $a+1$ entry and no $a$ entry.

    We must show that $\psi_1(r)$ is a Fibonacci ribbon. By Proposition~\ref{prop:ribboncondition}, it is sufficient to show that there are no $(1)(23)$ or $(12)(3)$ patterns in the columns of $\psi_1(r)$. Since $r_k$ is a Fibonacci ribbon, it is sufficient to show that if we add $(23)$ or $(3)$ as the leftmost column, the leftmost column of $r_k$ is not $(1)$ or $(12)$ respectively. We add $(23)$ to the left of $r_k$ only if $r_k(c_k)=(1)$ or $(13)$, i.e. $r_k(c_k)$ is 1-unbonded. As we showed above, if $r_k(c_k)$ is 1-unbonded, then $c_k$ is the maximal 1-unbonded column in $r_k$. As such, if the leftmost column of $r_k$ is $(1)$, then $c_k$ is the leftmost column of $r_k$. As $c_k\le c_{k-1}'$, it follows that $c_k=c_{k-1}'$. However, $r_k(c_{k-1}')\ne (1)$ by definition of $f$, so this is impossible. As such, if we add $(23)$ to $r_k$, the leftmost column of $r_k$ cannot be $(1)$. An analogous argument holds for $(12)$ and $(3)$, so we conclude that $\psi_1(r)$ is a Fibonacci ribbon as desired.

    Furthermore, let $d$ be the final column of $r_k$. Since $c_k$ is the maximal $a$-unbonded column in $r_k$, all $a$ entries in $r_k([c_k+1,d])$ must be bonded to $a+1$ entries in $r_k([c_k+1,d])$. Likewise, all $a+1$ entries in $r_k([c_k+1,d])$ must be bonded to $a$ entries in $r_k([c_k+1,d])$, as otherwise $c_k$ could not be $a$-unbonded by Proposition~\ref{prop:naturalmatch}. It follows that $\mu_a(r_k([c_k,d]))=\mu_{a+1}(r_k([c_k,d]))+1$. This is sufficient to conclude that adding a column to the end with an $a+1$ entry and no $a$ entry preserves ballot-admissibility, and so $\psi_1(r)$ is ballot-admissible.
    
    Then, we see that $c_k$ is bonded to $d+1$ in $\psi_1(r)$. Note that no other columns can be bonded to $d+1$ in $\psi_1(r)$. As such, all other bonds in $\psi_1(r)$ are defined based on the entries $\psi_1(r)([1,d])=r_k$ and so are the same as in $r_k$. 

    Finally, we show that $\psi_1(r)$ is atomic. By Proposition~\ref{prop:faddedbond}, $c_i$ and $c_i'$ are bonded in $r_{i+1}$ for all $i < k$. Now, $c_i < c_i'$, and $c_i' < c_{i+1}'$ by Proposition~\ref{prop:fterminate}, so $c_i, c_i' < c_j'$ for all $j\ge i+1$. By Proposition~\ref{prop:bondchangecrit}, this means that for all $j\ge i+1$, $c_i$ and $c_i'$ are bonded in $r_j$. As we showed above, adding a leftmost column to construct $\psi_1(r)$ does not change the bonds in the final $r_j$. We conclude that in $\psi_1(r)$, $c_i$ is bonded to $c_i'$ for all $i$ except for our final $c_i$, which is bonded to the $d+1$ in $\psi_1(r)$ as we showed above.
    
    Now, take some arbitrary column $c > 1$ of $\psi_1(r)$. If for some $i$, $c_i < c \le c_i'$, then we have a column $c_i < c$ bonded to a column $c_i'\ge c$ in $\psi_1(r)$. Otherwise, say that for all $i < k$, it is not the case that $c_i < c \le c_i'$. Then, for any $i < k$, if $c_i < c$, it follows that $c_i' < c$, so since $c_{i+1} \le c_i'$, $c_{i+1} < c$. We know that $c_0=1 < c$, so by induction, it follows that $c_{i+1} < c$ for all $0\le i < k$. Then, $c_k < c$, and we know that $c \le d+1$, so we have $c_k < c$ bonded to $d+1\ge c$ in $\psi_1(r)$. Thus, in all cases, we satisfy the criterion given in Proposition~\ref{prop:atomiccrit} and we conclude that $\psi_1(r)$ is atomic.
\end{proof}

\begin{definition}
    Let $r$ be an atomic ribbon of length $n+2$. We define $\psi_2(r)$ to be a ballot-admissible ribbon of length $n$ as follows. Remove the final, leftmost column of $r$ to construct $r_0$. Let $c_0$ be the column bonded to the final column in $\psi_2(r)$. Then, let $(r_{i+1},c_{i+1})=g(r_i,c_i)$. This continues until $c_i$ is the rightmost column of $r_i$. We then remove $c_i$ to construct $\psi_2(r)$.
\end{definition}

\begin{proposition}\label{prop:psi2welldefined}
    Let $r$ be an atomic ribbon of length $n+2$ for $n\ge 1$. Then, $\psi_2(r)$ is a well-defined ballot-admissible ribbon of length $n$.
\end{proposition}

\begin{proof}
    We construct $r_0$ by removing the leftmost column of $r$. Since $r$ is ballot-admissible, $r_0$ must be ballot-admissible by definition. In addition, we know that $r$ is atomic, so $r$ cannot end on $(1)$ or $(12)$ by definition. This means that the final column of $r$ must be bonded to some unique column $c_0$ in $r$ which then becomes unbonded in $r_0$. Observe that all other bonds of columns of $r$, which are not bonded to the final column, are unaffected by this. 

    In addition, say that $c_0\ne 1$. There exists some column preceding $c_0$ bonded to some column $\ge c_0$. Since all bonds not involving the final column of $r$ are the same in $r_0$, we see that this also holds in $r_0$. Then, if $r_0(c_0)$ is $a$-unbonded, any column $c > c_0$ in $r_0$ is $a$-unbonded iff $r(c)$ is $a$-unbonded. Since $r(c_0)$ is bonded to the final column of $r$, no such $r(c)$ can be $a$-unbonded, so $r_0(c_0)$ is the maximal $a$-unbonded column in $r_0$. From this, by Proposition~\ref{prop:gwelldefined} and Proposition~\ref{prop:gaddedbond}, $(r_1,c_1)=g(r_0,c_0)$ is well-defined, and either $r_1(c_0)$ is bonded or $c_0=c_0'$.

    Next, for some $i > 0$, assume that $(r_i,c_i)$ is well-defined, $r_i(c_{i-1})$ is bonded or $c_{i-1}=c_{i-1}'$, and $c_i\ne 1$. Then, we know that there exists some $c < c_i$ bonded to some column $d\ge c_i$ in $r$, since $r$ is atomic. Furthermore, by the same argument as above, $c$ remains bonded to $d$ in $r_0$. Now, suppose that $c$ is bonded to $d$ in $r_j$ for $0\le j < i$. By Proposition~\ref{prop:gterminate}, as $i\ge j+1$, $c_i \le c_{j+1}$, so as $c < c_i$, $c < c_{j+1}$. Then, as $c < c_{j+1}$ and $c$ is bonded to $d>c$ in $r_j$, $c$ and $d$ are also bonded in $r_{j+1}$ by Proposition~\ref{prop:gpreservation}. As $c$ and $d$ are bonded in $r_0$, it follows by induction that for all $0\le j < i$, $c$ and $d$ are bonded in $r_{j+1}$. So, $c$ and $d$ are bonded in $r_i$. Since $(r_i,c_i)=g(r_{i-1},c_{i-1})$, and $c < c_i$ is bonded in $r_i$ to $d\ge c_i$, by Proposition~\ref{prop:gdouble}, $g(r_i,c_i)=(r_{i+1},c_{i+1})$ is well-defined and either $r_{i+1}(c_i)$ is bonded or $c_i=c_i'$. Note that by Proposition~\ref{prop:gwelldefined}, all $r_i$ are ballot-admissible.
    
    By induction, we now see that until $c_i=1$, $(r_{i+1},c_{i+1})$ is well-defined with $r_{i+1}(c_i)$ being bonded. However, by Proposition~\ref{prop:gterminate}, $c_{i+1} < c_i$ for all $i$, so eventually for some $k$, $g(r_k,c_k)$ must not be well-defined. By our argument above, $c_k=1$. Then, since the initial column of $r_k$ is unbonded, while $r_k$ is ballot-admissible, we may simply remove the initial column of $r_k$ while preserving ballot-admissibility. Thus, we construct a ballot-admissible ribbon $\psi_2(r)$ as desired.
\end{proof}

\begin{lemma}\label{lemma:fdisallowedbonds}
    Define $r_i$, $c_i$, and $c_i'$ as in the construction of $\psi_1$. Let $i$ be arbitrary. For all $c < c_{i+1}$, $c$ cannot be bonded to a column $>c_i'$ in $r_{i+1}$. 
\end{lemma}

\begin{proof}
    Suppose that for some column $c$, $c < c_{i+1}$ and $c$ is bonded to some column $>c_i'$ in $r_{i+1}$ for some $i$. Then, by Proposition~\ref{prop:fuseful}, $c < c_i$. 
    
    Let $j$ be the minimal value such that $c < c_j$ and $c$ is bonded to some column $>c_j'$ in $r_{j+1}$. It is impossible for any $c$ that $c < c_0$, so $j\ge 1$. Now, by Proposition~\ref{prop:fpreservation}, $c$ is bonded to some column $\ge c_j' > c_{j-1}'$ in $r_j$. In addition, since $c < c_j$ and $c$ is bonded in $r_j$ to some column $>c_{j-1}'$, by Proposition~\ref{prop:fuseful}, $c < c_{j-1}$. As such, $c < c_{j-1}$ and $c$ is bonded to some column $>c_{j-1}'$ in $r_j$. However, by the minimality of $j$, this is impossible. As such, we have a contradiction, so any column $c < c_i$ cannot be bonded to a column $>c_i'$ in $r_{i+1}$ for any $i$.
\end{proof}

\begin{theorem}\label{thm:atomicbijection}
    The maps $\psi_1$ and $\psi_2$ are inverses which induce a bijection between ballot-admissible ribbons of length $n$ and atomic ribbons of length $n+2$ for all $n\ge 1$.
\end{theorem}

\begin{proof}
    First, we define $r_i$ and $c_i$ as in the construction of $\psi_2$. As we showed in Proposition~\ref{prop:psi2welldefined}, for all $i\ge 1$, given $(r_i,c_i)=g(r_{i-1},c_{i-1})$, $r_i(c_{i-1})$ is bonded or $c_{i-1}=c_{i-1}'$. By Proposition~\ref{prop:ginverse}, it follows that $f(r_i,c_i)=(r_{i-1},c_{i-1})$ for all $i\ge 1$. From this, it immediately follows that $\psi_1\circ\psi_2=\mathrm{id}$.

    Next, we show that $\psi_2\circ\psi_1=\mathrm{id}$. For the remainder of the proof, we use $r_i$, $c_i$, and $c_i'$ as defined in the construction of $\psi_1$. It is sufficient to show that for all $i\ge 0$, $g(r_{i+1},c_{i+1})=(r_i,c_i)$.

    First, suppose for the sake of contradiction that for some $i$, $c_{i+1} \le c_i$. We know that in $r_{i+1}$, $c_{i+1}$ is unbonded while $c_i$ is bonded, so $c_{i+1}\ne c_i$. As such, $c_{i+1} < c_i$. Then, let $c_{i+1}$ be bonded to $d$ in $r_i$. Since $c_{i+1}$ is bonded to $d$ in $r_i$ but is unbonded in $r_{i+1}$, we see that $d \ge c_i' > c_{i-1}'$ by Proposition~\ref{prop:bondchangecrit}. Then, we have that $c_{i+1} < c_i$ is bonded to some column $d > c_{i-1}'$ in $r_i$. However, by Lemma~\ref{lemma:fdisallowedbonds}, this is impossible, so we have a contradiction. From this, we conclude that for all $i$, $c_i < c_{i+1}$.

    Now, since $f(r_i,c_i)=(r_{i+1},c_{i+1})$, we know by Proposition~\ref{prop:faddedbond} that $c_i < c_{i+1}$ is bonded to $c_i'$ where $c_i'\ge c_{i+1}$ in $r_{i+1}$. To show that $g(r_{i+1},c_{i+1})=(r_i,c_i)$, it is sufficient to demonstrate that $c_i'$ is the largest column $\ge c_{i+1}$ bonded to some column preceding $c_{i+1}$ in $r_{i+1}$. In fact, this is exactly the statement of Lemma~\ref{lemma:fdisallowedbonds}. As such, we may conclude that $g(r_{i+1},c_{i+1})=(r_i,c_i)$. We have now shown that $\psi_1$ and $\psi_2$ are inverses and conclude that they induce a bijection.
\end{proof}

\begin{corollary}\label{cor:equinumerous}
    For all $n\ge 0$, ballot-admimssible ribbons of length $n$ are equinumerous with peakless Motzkin paths of length $n$.
\end{corollary}

\begin{proof}
    We show this by induction. For $n=0$, we say that there is a single empty ribbon and a single peakless Motzkin path with no steps. These are evidently equinumerous.
    
    Next, suppose that ballot-admissible ribbons of length $m$ are equinumerous with peakless Motzkin paths of length $m$ for all $0\le m < n$. Then, we will show that the same holds for $n$.

    For any ballot-admissible ribbon $r$ of length $n$, there exists some maximal column $c$ of $r$ such that $r'=r([c,n])$ is ballot-admissible. We will count ballot-admissible ribbons where $r'$ has length $j$, which we note is uniquely defined for each ribbon $r$. We note that for ribbons of length $n$, $1\le j \le n$. In addition, observe that $r'$ must be atomic, while no atomic ribbons of length 2 exist. It follows that $j\ne 2$. It is easy to see that choosing some ribbon $r$ for some fixed value of $j$ is equivalent to the following two chioces. First, we choose some arbitrary ballot-admissible ribbon of length $n-j$ to construct $r([1,c-1])$. Then, we choose some atomic ribbon of length $j$ for $r'$.

    Since ballot-admissible ribbons of length $n-j$ are equinumerous with peakless Motzkin paths of the same length, we have $a_{n-j}$ options for the first choice. Meanwhile, by Theorem~\ref{thm:atomicbijection}, atomic ribbons of length $j$ are equinumerous with ballot-admissible ribbons of length $j-2$, which are counted by $a_{j-2}$ for $3\le j\le n$. In total, this gives us $a_{n-j}*a_{j-2}$ choices for each $j$. Note that if $j=1$, we have exactly one atomic ribbon of length 1, so we have $a_{n-1}$ options. This gives us 
    \[a_{n-1} + \sum_{j=3}^{n} a_{j-2}*a_{n-j} = a_{n-1} + \sum_{j=1}^{n-2} a_j*a_{n-j-2}\]

    total ballot-admissible ribbons of length $n$. By Proposition~\ref{prop:motzkinrecurrence}, we have $a_n$ ballot-admissible ribbons of length $n$. So, ballot-admissible ribbons and peakless Motzkin paths of length $n$ are equinumerous.
\end{proof}

\begin{definition}
    Let $r$ be a ballot-admissible ribbon. We define a peakless Motzkin path $\phi(r)$ by mapping $r(i)$ to the $i$th step of $\phi(r)$ as follows:
    \begin{enumerate}
        \item[(i)] If $r(i)$ is a bonded $(12)$ or $(1)$ then the $i$th step of $\phi(r)$ is $N$.
        \item[(ii)] If $r(i)$ is an unbonded $(12)$ or $(1)$, or a bonded $(13)$ or $(2)$, then the $i$th step is $E$.
        \item[(iii)]  If $r(i)$ is an unbonded $(13)$ or $(2)$, or $(23)$ or $(3)$, then the $i$th step is $S$.
    \end{enumerate}

    Furthermore, let $\phi_n$ be $\phi$ restricted to ribbons of length $n$.
\end{definition}

\begin{lemma}\label{lemma:bijectionpreliminary}
    Let $r$ be a ballot-admissible ribbon. Then, $\phi(\psi_1(r))=N\phi(r) S$.
\end{lemma}

\begin{proof}
    We define $r_i$, $c_i$ and $c_i'$ as in the construction of $\psi_1$.
    
    First, we construct $r_0$ by adding $(12)$ to the beginning of $r$. This does not change the bonds in the columns of $r$, but adds an initial unbonded column to the beginning of $r$. As such, $\phi(r_0)=E\phi(r)$. 
    
    Next, consider that $(r_{i+1},c_{i+1})=f(r_i,c_i)$. We know that only the entries of $c_i'$ change in this step. Also, by Proposition~\ref{prop:fremovedbond} and Proposition~\ref{prop:faddedbond}, the only columns which go from bonded to unbonded or vice versa are $c_i$, $c_i'$, and $c_{i+1}$. As such, in all but the steps corresponding to these columns, $\phi(r_i)$ and $\phi(r_{i+1})$ agree. Since $c_i$ goes from unbonded in $r_i$ to bonded in $r_{i+1}$, the $c_i$-th step goes from $E$ or $S$ to $N$ or $E$. 

    Now, consider the case where $r_i(c_i')$ is an unbonded $(1)$ or $(12)$ while $r_{i+1}(c_i')$ is a bonded $(2)$ or $(13)$. Here, we see that the $c_i'$-th step is $E$ in both $\phi(r_i)$ and $\phi(r_{i+1})$. Meanwhile, we see that $c_{i+1} < c_i'$ is the unique column which goes from bonded in $r_i$ to unbonded in $r_{i+1}$. As such, the $c_{i+1}$-th step is $N$ or $E$ in $\phi(r_i)$ but $E$ or $S$ in $\phi(r_{i+1})$. The same holds in the case where $r_i(c_i')$ is an unbonded $(2)$ or $(13)$ and $r_{i+1}(c_i')$ is $(3)$ or $(23)$.
    Finally, consider the case where $r_i(c_i')$ is an unbonded $(1)$ or $(12)$ and $r_{i+1}(c_i')$ is an unbonded $(2)$ or $(13)$. Then, we see that the $c_i'$-th step goes from $E$ in $\phi(r_i)$ to $S$ in $\phi(r_{i+1})$. In addition, by definition of $f$, $c_{i+1}=c_i'$.

    In all cases, we conclude that exactly the steps corresponding to $c_i$ and $c_{i+1}$ change. We have $c_i$ going from $E$ or $S$ to $N$ or $E$, and the opposite for $c_{i+1}$. Let $c_j$ be the final $c_i$. We see that for $1\le i < j$, this process alters the step corresponding to $c_i$ exactly twice. First, the $c_i$-th step goes from $N$ or $E$ in $\phi(r_{i-1})$ to $E$ or $S$ in $\phi(r_i)$. Then, the same step goes from $E$ or $S$ in $\phi(r_i)$ to $N$ or $E$ in $\phi(r_{i+1})$. In this way, from $r_0$ to $r_j$, the $c_i$-th step for $1\le i < j$ goes not change overall. Meanwhile, the $c_0$-th step changes from $E$ to $N$, and the $c_j$-th step goes from $N$ or $E$ to $E$ or $S$. Finally, to construct $\psi_1(r)$, we add a new leftmost column which bonds to $c_j$ and does not alter any other bonds. In this way, the $c_j$-th step goes from $E$ or $S$ in $r_j$ to $N$ or $E$ in $\psi_1(r)$ and overall is not altered. Furthermore, our final column in $\psi_1(r)$ is either an unbonded $(2)$ or $(13)$, or it is $(23)$ or $(3)$. As such, we conclude that $\phi(\psi_1(r))=N\phi(r)S$ as desired.
\end{proof}

\begin{proposition}\label{prop:phiwelldefined}
    Let $r$ be a ballot-admissible ribbon of length $n$. Then, $\phi(r)$ is a well-defined peakless Motzkin path of length $n$.
\end{proposition}

\begin{proof}
    We will show this by induction on $n$. Observe that $\phi$ maps $r=(1)$ to $\phi(r)=E$, which is the unique peakless Motzkin path of length 1. As such, this holds for $n=1$.

    Now, suppose that for all $1\le m < n$, $\phi_m$ is a map from ballot-admissible ribbons to peakless Motzkin paths of length $m$. We will show that $\phi_n$ is consequently also a map from ballot-admissible ribbons to peakless Motzkin paths. Note that by definition of $\phi$, it is easy to see that $\phi$ maps ballot-admissible ribbons to lattice walks generated by $N$, $E$, and $S$.

    Let $r$ be some arbitrary ballot-admissible ribbon of length $n$. Then, first consider the case where $r$ is not atomic. We know that $r=r_1 r_2$ for some unique ballot-admissible $r_1$ and atomic $r_2$. In addition, since $r$ is not atomic, both $r_1$ and $r_2$ have lengths $ < n$. Then, it is easy to see that $\phi(r)=\phi(r_1)\phi(r_2)$. By our assumption, since $r_1$ and $r_2$ have lengths $<n$, $\phi(r_1),\phi(r_2)$ are peakless Motzkin paths. Then, they cannot begin on $S$ or end on $N$, since they must begin and end at $y=0$ without going below $y=0$. It is now easy to conclude that $\phi(r)=\phi(r_1)\phi(r_2)$ has no $NS$ patterns and is a peakless Motzkin path.

    Meanwhile, consider the case where $r$ is atomic. We assume that $1 < n$, so as there are no atomic ribbons of length 2, $n\ge 3$. Then, by Lemma~\ref{lemma:bijectionpreliminary}, $\phi(r)=N\phi(\psi_2(r))S$. We note that $\psi_2(r)$ is a ballot-admissible ribbon of length $n-2 < n$, so $\phi(\psi_2(r))$ must be a peakless Motzkin path. It then immediately follows that $\phi(r)=N\phi(\psi_2(r))S$ is also a peakless Motzkin path.

    As such, in all cases, we conclude that $\phi_n(r)$ is a peakless Motzkin path as desired. Note that by definition of $\phi$, it is immediately obvious that $\phi_n(r)$ has length $n$.
\end{proof}

\begin{theorem}
    The map $\phi$ induces a bijection from ballot-admissible ribbons of length $n$ to peakless Motzkin paths of length $n$ for all $n\ge 1$.
\end{theorem}

\begin{proof}
    By Proposition~\ref{prop:phiwelldefined}, we know that $\phi$ is a map from ballot-admissible ribbons of length $n$ to peakless Motzkin paths of length $n$. Moreover, by Corollary~\ref{cor:equinumerous}, we know that ballot-admissible ribbons of length $n$ are equinumerous with peakless Motzkin paths of length $n$. As such, to show that $\phi_n$ is a bijection, it is sufficient to show that $\phi_n$ is injective.

    We will show this by induction. For $n=1$, it is immediately obvious that $\phi_1$ is bijective. Then, suppose that for $1\le m < n$, all $\phi_m$ are bijective. 
    
    Let $r$ be some ballot-admissible ribbon of length $n$. Then, there exists some unique maximal column $c$ such that the ribbon $r'=r([c,n])$ is ballot-admissible. We say that $r=r'' r'$, where $r''=r([1,c-1])$. It is easy to see that $r'$ is atomic and that $r''$ is ballot-admissible. There are no atomic ribbons of length 2, so either $r'$ is of length 1 or $r'$ is of length $j\ge 3$.

    If $r'$ is of length 1, then $r'=(1)$ or $(12)$, so we find that $\phi(r)=\phi(r'')E$. Meanwhile, if $r'$ is of length $j\ge 3$, then $\phi(r)=\phi(r'')\phi(r')=\phi(r'')N\phi(\psi_2(r'))S$ by Lemma~\ref{lemma:bijectionpreliminary}. This gives us a characterization which we will use to show that $\phi_n$ is injective.

    Suppose that for $r_1,r_2$, $p=\phi_n(r_1)=\phi_n(r_2)$. Then, either $p$ ends on $E$ or $S$. If $p$ ends on $E$, then by our characterization above, $r_1'=r_2'$ and $\phi(r_1'')E=\phi(r_2'')E$. Then, $\phi(r_1'')=\phi(r_2'')$. Since $r_1''$ and $r_2''$ have length $n-1 < n$, and we assume that $\phi_{n-1}$ is injective, we conclude that $r_1''=r_2''$. It then immediately follows that $r_1=r_2$.

    Next, suppose that $p$ ends on $S$. Then, it is easy to show that there exist some unique peakless Motzkin paths $p_1,p_2$ such that $p=p_1 N p_2 S$ by taking the last $N$ step at $y=0$ in $p$. By our characterization above, $p=\phi(r_1'')N\phi(\psi_2(r_1'))S=\phi(r_2'')N\phi(\psi_2(r_2'))S$. It follows that $p_1=\phi(r_1'')=\phi(r_2'')$ and likewise that $p_2=\phi(\psi_2(r_1'))=\phi(\psi_2(r_2'))$. Then, since $r_1''$, $r_2''$, $\psi_2(r_1')$ and $\psi_2(r_2')$ all have length $< n$, and for all $1\le m < n$, $\phi_m$ is bijective, it follows that $r_1''=r_2''$ and that $\psi_2(r_1')=\psi_2(r_2')$. Then, as $\psi_2$ is injective, $r_1'=r_2'$. Since $r_1'=r_2'$ and $r_1''=r_2''$, $r_1=r_2$ as desired.

    Thus, we have shown for arbitrary $r_1$ and $r_2$ that if $\phi_n(r_1)=\phi_n(r_2)$, then $r_1=r_2$. We conclude that $\phi_n$ is injective, so it is also a bijection.
\end{proof}

\subsection{Examples}
\begin{example}
    We give an example of a tableau with a linear form to illustrate our definition for the linear form. Take the tableau seen below. We see that the rightmost column has entries 1 and 2, so we denote it $(12)$. The second column from the right is likewise $(1)$. Continuing with this process, we find that our final linear form is $(12)(1)(13)(2)$.
    \begin{center}
        \begin{ytableau}
    \none & \none & \none & {\color{black} 1}\\
    \none & {\color{black} 1} & {\color{black} 1} & {\color{black} 2}\\
    {\color{black} 2} & {\color{black} 3} & \none & \none
    \end{ytableau}
    \end{center}
\end{example}

\begin{example}
    We demonstrate that $\psi_1((1)(12)(2))=(12)(1)(13)(2)(23)$. We outline the process with visualizations of each ribbon. Note that bonded entries are given the same color.
    
    First, we add a column $(12)$ to the right of $r$ in order to construct $r_0$. Then, we find that the first unbonded $(12)$ or $(2)$ after our initial column is the third column, so we change this column from $(12)$ to $(13)$. Now, this leaves our second column unbonded, and we find that there are no subsequent unbonded $(1)$ or $(13)$ columns, so we add $(23)$ to the end as the leftmost column. This produces the process below:
    
    
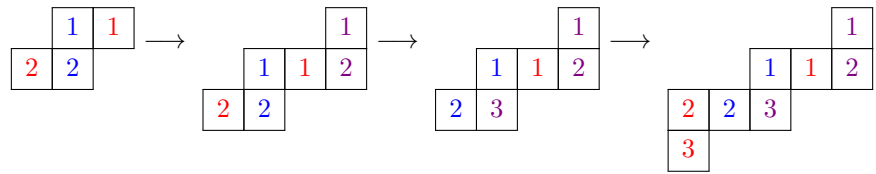
\begin{center}
    \begin{ytableau}
    \none & {\color{blue} 1} & {\color{red} 1}\\
    {\color{red} 2} & {\color{blue} 2} & \none
    \end{ytableau} $\longrightarrow{}$ \begin{ytableau}
    \none & \none & \none & {\color{violet} 1}\\
    \none & {\color{blue} 1} & {\color{red} 1} & {\color{violet} 2}\\
    {\color{red} 2} & {\color{blue} 2} & \none & \none
    \end{ytableau} $\longrightarrow{}$ 
    \begin{ytableau}
    \none & \none & \none & {\color{violet} 1}\\
    \none & {\color{blue} 1} & {\color{red} 1} & {\color{violet} 2}\\
    {\color{blue} 2} & {\color{violet} 3} & \none & \none
    \end{ytableau} $\longrightarrow{}$ 
    \begin{ytableau}
    \none & \none & \none & \none & {\color{violet} 1}\\
    \none & \none & {\color{blue} 1} & {\color{red} 1} & {\color{violet} 2}\\
    {\color{red} 2} & {\color{blue} 2} & {\color{violet} 3} & \none & \none\\
    {\color{red} 3} & \none & \none & \none & \none\\
    \end{ytableau}
    \end{center}

    Note further that if we apply $\phi$ to each step in this process, we get $NES$, $ENES$, $NEES$, $NNESS$.
\end{example}

\begin{example}
    We demonstrate that $\phi((12)(1)(13)(2)(12)(2)(13)(3))=NNESNESS$. Consider the visualization of this ribbon below. Using the colors that illustrate bonds, it is easy to see which columns are bonded to others. Upon inspection, the desired result is obtained for $\phi$.

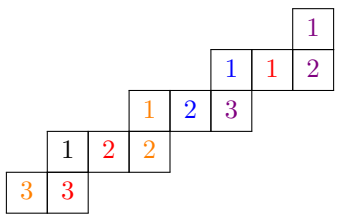
    \begin{center}
    \begin{ytableau}
    \none & \none & \none & \none & \none & \none & \none & {\color{violet} 1}\\
    \none & \none & \none & \none & \none & {\color{blue} 1} & {\color{red} 1} & {\color{violet} 2}\\
    \none & \none & \none & {\color{orange} 1} & {\color{blue} 2} & {\color{violet} 3} & \none & \none\\
    \none & 1 & {\color{red} 2} & {\color{orange} 2} & \none & \none & \none & \none\\
    {\color{orange}3} & {\color{red}3}\\
    \end{ytableau}
    \end{center}
\end{example}

\section{Conclusion and Future Work}\label{sec:future}
We have constructed a bijection between certain ballot-admissible Fibonacci ribbon tableaux, described in~\cite{tableauconjecturepaper} as the set of ballot-admissible $\text{Fib}(3,k)$ ribbons, and peakless Motzkin paths. However, the problem of enumerating ballot-admissible $\text{Fib}(n,k)$ ribbons for $n\ge 4$ remains an open problem. This paper has introduced a variety of techniques which may be helpful in order to establish a recursive enumeration of ballot-admissible $\text{Fib}(n,k)$ ribbons in the general case. In addition, it is of interest whether a bijection exists between $\text{Fib}(n,k)$ ribbons for $n\ge 4$ and other classes of Motzkin paths with forbidden patterns.

\section*{Acknowledgements}
This work was completed when the author was visiting Yale University for the SUMRY REU. The author would like to thank SUMRY, and particularly Professor Asamoah Nkwanta. The author would also like to thank Swarthmore College for funding this work through the Allen and Naomi Schneider Summer Research Fund. 

\bibliographystyle{IEEEtran}
\bibliography{motzkin_tableau_bijection}

\end{document}